\documentclass[10pt,a4paper]{amsart}

\usepackage[dvipsnames]{xcolor} 
\usepackage{amsmath,amsthm,amssymb}
\usepackage{enumitem}
\usepackage{tikz}
\usepackage{hyperref}
\usepackage{amsfonts,marvosym,amscd}
\usepackage{mathtools}
\usepackage{mathscinet}
\usepackage{wasysym} 
\usepackage{graphicx}
\usepackage{psfrag}
\usepackage{datetime}
\usepackage{tikz-cd}
\usepackage{url}
\usepackage{adjustbox}

\DeclareFontFamily{U}{mathx}{\hyphenchar\font45}
\DeclareFontShape{U}{mathx}{m}{n}{
 <5> <6> <7> <8> <9> <10>
 <10.95> <12> <14.4> <17.28> <20.74> <24.88>
 mathx10
 }{}
\DeclareSymbolFont{mathx}{U}{mathx}{m}{n}
\DeclareFontSubstitution{U}{mathx}{m}{n}
\DeclareMathAccent{\widecheck}{0}{mathx}{"71}

\usepackage{tabstackengine}
\setstackgap{L}{.75\baselineskip}

\usetikzlibrary{decorations.pathmorphing}
\usetikzlibrary{positioning}
\usetikzlibrary{calc}
\usetikzlibrary{intersections}

\numberwithin{equation}{section}
\newtheorem{theorem}{Theorem}[section]
\newtheorem{lemma}[theorem]{Lemma}
\newtheorem{proposition}[theorem]{Proposition}
\newtheorem{corollary}[theorem]{Corollary}

\newtheorem{innercustomthm}{Theorem}
\newenvironment{customthm}[1]
 {\renewcommand\theinnercustomthm{#1}\innercustomthm}
 {\endinnercustomthm}

\theoremstyle{definition}

\newtheorem{example}[theorem]{Example}
\newtheorem{convention}[theorem]{Convention}
\newtheorem{construction}[theorem]{Construction}

\theoremstyle{remark}
\newtheorem{remark}[theorem]{Remark}
\newtheorem{observation}[theorem]{Observation}

\newcommand{\subs}[3]{\binom{[#1]}{#2}_{#3}}
\newcommand{\floor}[1]{\lfloor #1 \rfloor}
\newcommand{\ceil}[1]{\lceil #1 \rceil}

\newcommand{\simp}{\mathring{e}}

\newcommand{\intsp}{\mathrm{ISp}}

\let\emptyset\varnothing

\makeatletter
\newcommand\@dotsep{4.5}
\def\@tocline#1#2#3#4#5#6#7{\relax
  \ifnum #1>\c@tocdepth % then omit
  \else
    \par \addpenalty\@secpenalty\addvspace{#2}%
    \begingroup \hyphenpenalty\@M
    \@ifempty{#4}{%
      \@tempdima\csname r@tocindent\number#1\endcsname\relax
    }{%
      \@tempdima#4\relax
    }%
    \parindent\z@ \leftskip#3\relax \advance\leftskip\@tempdima\relax
    \rightskip\@pnumwidth plus1em \parfillskip-\@pnumwidth
    #5\leavevmode\hskip-\@tempdima{#6}\nobreak
    \leaders\hbox{$\m@th\mkern \@dotsep mu\hbox{.}\mkern \@dotsep mu$}\hfill
    \nobreak
    \hbox to\@pnumwidth{\@tocpagenum{\ifnum#1=1\fi#7}}\par% 
    \nobreak
    \endgroup
  \fi}
\AtBeginDocument{%
\expandafter\renewcommand\csname r@tocindent0\endcsname{0pt}
}
\def\l@subsection{\@tocline{2}{0pt}{2.5pc}{5pc}{}}
\makeatother

\usepackage[backend=biber,style=alphabetic,url=false,maxbibnames=99]{biblatex}
\DeclareFieldFormat*{title}{#1}
\renewbibmacro{in:}{}
\title[Quotients of the higher Bruhat orders]{The first higher Stasheff--Tamari orders are quotients of the higher Bruhat orders}
\author{Nicholas J. Williams}
\address{School of Mathematics and Actuarial Science, University of Leicester, University Road, LE1 7RH}
\email{njw40@le.ac.uk}

\begin{document}

\begin{abstract}
We prove the conjecture that the higher Tamari orders of Dimakis and M\"uller-Hoissen coincide with the first higher Stasheff--Tamari orders. To this end, we show that the higher Tamari orders may be conceived as the image of an order-preserving map from the higher Bruhat orders to the first higher Stasheff--Tamari orders. This map is defined by taking the first cross-section of a cubillage of a cyclic zonotope. We provide a new proof that this map is surjective and show further that the map is full, which entails the aforementioned conjecture. We explain how order-preserving maps which are surjective and full correspond to quotients of posets. Our results connect the first higher Stasheff--Tamari orders with the literature on the role of the higher Tamari orders in integrable systems.
\end{abstract}

\maketitle

\tableofcontents

\section{Introduction}

Two of the best known and most widely studied partially ordered sets in mathematics are the Tamari lattice \cite{tamari} and the weak Bruhat order on the symmetric group. The Tamari lattice appears in a broad range of areas of mathematics, physics, and computer science \cite{tamari-festschrift}. It was introduced by Tamari \cite{tamari,huang-tamari} as an order on the set of bracketings of a string. It is the 1-skeleton of the associahedron, which was famously used by Stasheff in topology to define $A_{\infty}$-spaces \cite{stasheff}. The weak Bruhat order on the symmetric group was first studied by statisticians in the 1960s \cite{savage,lehmann,yo_bruhat} and is now a fundamental part of Coxeter theory---see \cite{bb_coxeter}, for example. It furthermore provides a useful framework for studying questions in the theory of social choice \cite{abello_thesis,chameni-nembua,abello_bruhat,dkk_condorcet_tiling}. Both orders also appear in the representation theory of algebras \cite{buan-krause,thomas_tamari,mizuno-preproj,irrt} and the theory of cluster algebras \cite{fz-y,rs_framework}.

These posets have higher-dimensional versions, namely the first higher Stasheff--Tamari orders $\mathcal{S}(n, \delta)$ \cite{kv-poly,er} and the higher Bruhat orders $\mathcal{B}(n, \delta + 1)$ \cite{ms}. These higher posets arise as orders on (equivalence classes of) maximal chains in the original posets, and then as orders on (equivalence classes of) maximal chains in those posets, and so on \cite{ms,rambau}. For instance, the elements of the higher Bruhat order $\mathcal{B}(n,2)$ correspond to reduced expressions for the longest element in the symmetric group, while the covering relations correspond to braid moves. In this way, the higher posets encode higher-categorical data latent within the original posets. Another way of thinking of the higher-dimensional posets is geometrically. The Tamari lattice concerns triangulations of convex polygons, whereas the first higher Stasheff--Tamari orders concern triangulations of cyclic polytopes; the weak Bruhat order concerns complexes of edges of hypercubes, whereas the higher Bruhat orders concern complexes of faces of hypercubes, known as \emph{cubillages} or \emph{fine zonotopal tilings}.

The first higher Stasheff--Tamari orders and the higher Bruhat orders have their own connections with other areas of mathematics. The first higher Stasheff--Tamari orders occur in the representation theory of algebras \cite{njw-hst} and algebraic $K$-theory \cite{poguntke}. The higher Bruhat orders were originally introduced to study hyperplane arrangements \cite{ms} and have found application in the theories of Soergel bimodules \cite{elias_bruhat}, quasi-commuting Pl\"ucker coordinates \cite{lz}, and social choice \cite{gr_bruhat}. They are also tightly connected with the quantum Yang--Baxter equation and its generalisations \cite[and references therein]{dm-h-simplex}.

The relation between the Tamari lattice and the weak Bruhat order has been of significant interest. There is a classical surjection from the latter to the former, which can be realised as a map from permutations to binary trees. This map arises in many different places \cite{bw_coxeter,bw_shell_2,tonks,lr_hopf,lr_order,reading_cambrian}. Kapranov and Voevodsky extended this surjection to a map from the higher Bruhat orders to the first higher Stasheff--Tamari orders $f \colon \mathcal{B}(n, \delta) \to \mathcal{S}(n + 2, \delta + 1)$ \cite{kv-poly}, which they conjectured was a surjection as well. This remains an open problem despite some detailed studies \cite{rambau,thomas-bst}.

In this paper, we consider a closely related map from the higher Bruhat orders to the first higher Stasheff--Tamari orders $g\colon\mathcal{B}(n,\delta+1) \rightarrow \mathcal{S}(n,\delta)$. This map was first considered as a map of posets in \cite{thomas-bst}, in its dual form, and was itself considered in \cite[Appendix B]{dkk-survey}. As a map of sets it was considered and shown to be surjective in \cite{rs-baues}, using the language of \emph{lifting triangulations}. We provide a new proof of surjectivity, and go further by showing that the map is full. We call order-preserving maps which are both surjective and full \emph{quotient maps of posets}.

\begin{customthm}{A}[Theorem~\ref{thm:quot}]\label{thm:int_a}
The map $g \colon\mathcal{B}(n,\delta+1) \rightarrow \mathcal{S}(n,\delta)$ is a quotient map of posets.
\end{customthm}

Indeed, in this paper we give a new approach to quotients of posets. The quotient of a poset by an arbitrary equivalence relation is not always a well-defined poset. Previous authors \cite{hs-char-poly,cs_cong,reading_order} have given sufficient conditions for the quotient to be well-defined which ensure that other structure is also preserved, such as lattice-theoretic properties. Because the posets we are considering are not in general lattices \cite{ziegler-bruhat,njw-hst}, we instead consider weaker conditions, which are necessary and sufficient for the quotient to be a well-defined poset. We show that quotients of posets in this sense correspond to order-preserving maps which are surjective and full.

Part of the motivation for considering quotient posets in the way that we do stems from \cite{dm-h}, where Dimakis and M\"uller-Hoissen apply an equivalence relation to the higher Bruhat orders to define the ``higher Tamari orders'' in order to describe a class of soliton solutions of the KP equation. In subsequent work \cite{dm-h-simplex}, the authors make further connections with mathematical physics by using the higher Tamari orders to define \emph{polygon equations}, an infinite family of equations which generalise the pentagon equation. The pentagon equation appears in many different areas of physics, including the theory of angular momentum in quantum mechanics \cite{dmh_amqp}, and conformal field theory \cite{dmh_cft}, as well as several other places \cite{dmh_qha,dmh_rd,dmh_qd}. The polygon equations which generalise the pentagon equation themselves occur in category theory \cite{kv-zam,street-fusion} and as ``Pachner relations'' in 4D topological quantum field theory \cite{dmh_tqft}.

Dimakis and M\"uller-Hoissen conjectured the higher Tamari orders to coincide with the first higher Stasheff--Tamari orders. We prove this conjecture by showing that the higher Tamari orders are given by the image of the map $g$, as first noted in \cite[Appendix B]{dkk-survey}. We then apply Theorem~\ref{thm:int_a}; for the two sets of orders to be equal, it is only necessary for the map $g\colon \mathcal{B}(n, \delta+1) \to \mathcal{S}(n, \delta)$ to be a quotient map of posets in our sense, rather than in any stronger sense. The upshot of our result is that two far-reaching sets of combinatorics are united. We unite the first higher Stasheff--Tamari orders, with their connections to the representation theory of algebras \cite{njw-hst}, and the higher Tamari orders, which describe classes of KP solitons \cite{dm-h} and from which arise the polygon equations \cite{dm-h-simplex}. Since the map $g$ is defined by taking a certain cross-section of a cubillage, our work shows the connection between \cite{dm-h} and the papers \cite{kk,gpw}, in which KP solitons are related to cross-sections of three-dimensional cubillages, building on \cite{kw_invent,kw_advances,huang_thesis}. See also \cite{galashin,olarte_santos}, for more work on cross-sections of cubillages.

\begin{customthm}{B}[Corollary~\ref{cor:t=st}]
The higher Tamari orders and the first higher Stasheff--Tamari orders coincide.
\end{customthm}

Our approach is to use the description of the higher Bruhat orders in terms of cubillages of cyclic zonotopes and the description of these objects in terms of separated collections established in \cite{gp} and studied extensively in \cite{dkk,dkk-interrelations,dkk-survey,dkk-weak,dkk-symmetric}. These tools allow us to construct cubillages which are pre-images under the map $g$, which is instrumental in the proof of Theorem~\ref{thm:int_a}.

This paper is structured as follows. In Section~\ref{sect:notation}, we lay out some notation and conventions that we use in the paper. We give background on the higher Bruhat orders and cubillages of cyclic zonotopes in Section~\ref{sect-hbo} and on the higher Stasheff--Tamari orders and triangulations of cyclic polytopes in Section~\ref{sect-hst}. In Section~\ref{sect:g_interpretations} we consider the map $g \colon \mathcal{B}(n, \delta + 1) \to \mathcal{S}(n, \delta)$. We give three different characterisations of this map in Sections~\ref{sect:g_geom},~\ref{sect:g_comb}, and~\ref{sect:vis}, which correspond to the three different possible interpretations of the higher Bruhat orders. In Section~\ref{sect:quotient_framework}, we lay the necessary groundwork in the theory of quotient posets to make the statement that $g$ is a quotient map of posets precise. In Section~\ref{sect-surj} we give a new proof of the fact that the map $g$ is surjective. We prove in Section~\ref{sect-quot} that it is full, and hence a quotient map of posets.

\subsection*{Acknowledgements}

This paper forms part of my PhD studies. I would like to thank my supervisor Professor Sibylle Schroll for her continuing support and attention. I would also like to thank Mikhail Kapranov for a clarification, Jordan McMahon for helpful comments on an earlier version of this paper, and Hugh Thomas and Mikhail Gorsky for interesting discussions. I am supported by a studentship from the University of Leicester. 

\section{Terminology and conventions}\label{sect:notation}

Here we outline some general terminology and conventions that we use throughout the paper.

\subsubsection{Notation}

We use $[n]$ to denote the set $\left\lbrace 1, \dots, n\right\rbrace $ and $\subs{n}{k}{}$ to denote the subsets of $[n]$ of size $k$. We sometimes refer to such subsets as $k$-subsets. Given a set $A \subseteq [n]$ such that $\# A = k + 1$, unless otherwise indicated, we shall denote the elements of $A$ by $A = \{a_{0}, \dots, a_{k}\}$, where $a_{0} < \dots < a_{k}$. The same applies to other letters of the alphabet: the upper case letter denotes the set; the lower case letter is used for the elements, which are ordered according to their index starting from 0.

\subsubsection{Ordering}

In this paper, it is convenient for us to consider both the linear and cyclic orderings of $[n]$. Unless stated otherwise, it should be assumed that we refer to the linear ordering on this set.

We denote by $(a,b), [a,b] \subseteq [n]$ respectively the \emph{open} and \emph{closed cyclic intervals}. That is,
\begin{align*}
(a,b) &:= \{ i \in [n] \mid a < i < b \text{ is a cyclic ordering} \}, \\
[a,b] &:= (a,b) \cup \{a,b\}.
\end{align*}
The one exception to this is that we will find it convenient to set $[a, a - 1] := \emptyset$. When we have $a<b$ in the linear ordering on $[n]$, we say that $[a,b]$ and $(a,b)$ are \emph{intervals}. We call $I \subseteq [n]$ an \emph{$l$-ple interval} if it can be written as a union of $l$ intervals, but cannot be written as a union of fewer than $l$ intervals. We similarly define \emph{cyclic $l$-ple intervals}.

When we refer to the elements $a_{i}$ of a subset $A \subseteq [n]$ with $\# A = d + 1$, we will sometimes write $i \in \mathbb{Z}/(d+1)\mathbb{Z}$ to indicate that one should interpret $a_{d + 1}$ as being equal to $a_{0}$. That is, if $A = \{1, 3, 5\}$, then $a_{0} = 1, a_{1} = 3, a_{2} = 5, a_{3} = 1$.

Given a linearly ordered set $L$ and $S \subset L$, we say that an element $l \in L \setminus S$ is an \emph{even gap} in $S$ if $\#\{s \in S \mid s > l \}$ is even. Otherwise, it is an \emph{odd gap}. A subset $S \subset L$ is \emph{even} if every $l \in L \setminus S$ is an even gap. A subset $S \subset L$ is \emph{odd} if every $l \in L \setminus S$ is an odd gap.

Let $\mathcal{P}$ be a partially ordered set. We say that $q$ \emph{covers} $p$ in $\mathcal{P}$ if $p < q$ and whenever $p \leqslant r \leqslant q$ in $\mathcal{P}$, then $r = p$ or $r = q$. If $q$ covers $p$ in $\mathcal{P}$ then we write $p \lessdot q$. If $\mathcal{P}$ is a finite poset, we have that $\mathcal{P}$ is the transitive-reflexive closure of its covering relations. Hence, in this case one can define $\mathcal{P}$ by specifying its covering relations.

\subsubsection{Convex geometry}

Recall that a set $\Gamma \subseteq \mathbb{R}^{\delta}$ is \emph
{convex} if, for any $\mathbf{x}, \mathbf{y} \in \Gamma$, the line segment $\overline{\mathbf{xy}}$ between $\mathbf{x}$ and $\mathbf{y}$ is contained in $\Gamma$. Given a set of points $\Gamma \subseteq \mathbb{R}^{\delta}$, the \emph{convex hull} $\mathrm{conv}(\Gamma)$ is defined to be the smallest convex set containing $\Gamma$ or, equivalently, the intersection of all convex sets containing $\Gamma$.

A \emph{convex polytope} is the convex hull of a finite set of points in $\mathbb{R}^{\delta}$. Let $\Delta \subset \mathbb{R}^{\delta}$ be a convex polytope. A \emph{facet} of $\Delta$ is a face of codimension one. The \emph{upper} facets of $\Delta$ are those that can be seen from a very large positive $\delta$-th coordinate. The \emph{lower} facets of $\Delta$ are those that can be seen from a very large negative $\delta$-th coordinate. A \emph{$k$-face} of a polytope is a face of dimension $k$. A \emph{subcomplex} of a polytope is a union of faces of the polytope.

Recall that for $\Gamma, \Gamma' \subseteq \mathbb{R}^{\delta}$, the \emph{Minkowski sum} of $\Gamma$ and $\Gamma'$ is defined to be \[\Gamma + \Gamma' = \{\mathbf{x} + \mathbf{y} \mid \mathbf{x} \in \Gamma,\, \mathbf{y} \in \Gamma'\}.\]

\section{Background}\label{sect-background}

\subsection{Higher Bruhat orders}\label{sect-hbo}

In this section we give the definition of the higher Bruhat orders. The fundamental definition of the higher Bruhat orders for our purposes is the description in terms of cubillages of cyclic zonotopes given in \cite{kv-poly} and formalised in \cite{thomas-bst}. After giving this definition, we give the characterisation of cubillages of cyclic zonotopes established in \cite{gp} and studied in \cite{dkk}. Finally, we explain the original definition of the higher Bruhat orders from \cite{ms}, which we will also need.

\subsubsection{Cubillages}

We first give the geometric description of the higher Bruhat orders due to \cite{kv-poly,thomas-bst}. Consider the \emph{Veronese curve} $\xi\colon \mathbb{R} \rightarrow \mathbb{R}^{\delta+1}$, given by $\xi_{t}=(1,t, \dots, t^{\delta})$. Let $\left\lbrace t_{1}, \dots, t_{n}\right\rbrace \subset \mathbb{R}$ with $t_{1}<\dots<t_{n}$ and $n \geqslant \delta+1$. The \emph{cyclic zonotope} $Z(n,\delta+1)$ is defined to be the Minkowski sum of the line segments \[\overline{\mathbf{0}\xi_{t_{1}}} + \dots + \overline{\mathbf{0}\xi_{t_{n}}},\] where $\mathbf{0}$ is the origin. The properties of the zonotope do not depend on the exact choice of $\{t_{1}, \dots, t_{n}\} \subset \mathbb{R}$. Hence, for ease we set $t_{i}=i$. For $k \geqslant l$ we have a canonical projection map
\begin{align*}
\pi_{k,l} \colon \mathbb{R}^{k} &\to \mathbb{R}^{l} \\
(x_{1}, \dots, x_{k}) &\mapsto (x_{1}, \dots, x_{l})
\end{align*}
which maps $Z(n, k) \to Z(n, l)$.

A \emph{cubillage} $\mathcal{Q}$ of $Z(n,\delta+1)$ is a subcomplex of $Z(n, n)$ such that $\pi_{n, \delta+1} \colon \mathcal{Q} \to Z(n, \delta +1)$ is a bijection. Note that $\mathcal{Q}$ therefore contains faces of $Z(n, n)$ of dimension at most $\delta+1$. We call these $(\delta + 1)$-dimensional faces of $\mathcal{Q}$ the \emph{cubes} of the cubillage. In the literature, cubillages are often called \emph{fine zonotopal tilings}---for example, in \cite{gp}.

After \cite[Theorem 4.4]{kv-poly} and \cite[Theorem 2.1, Proposition 2.1]{thomas-bst} one may define the \emph{higher Bruhat poset} $\mathcal{B}(n,\delta+1)$ as follows. The elements of $\mathcal{B}(n,\delta+1)$ consist of cubillages of $Z(n,\delta+1)$. The covering relations of $\mathcal{B}(n,\delta+1)$ are given by pairs of cubillages $\mathcal{Q} \lessdot \mathcal{Q}'$ where there is a $(\delta + 2)$-face $\Gamma$ of $Z(n, n)$ such that $\mathcal{Q} \setminus \Gamma = \mathcal{Q}' \setminus \Gamma$ and $\pi_{n, \delta + 2}(\mathcal{Q})$ contains the lower facets of $\pi_{n, \delta + 2}(\Gamma)$, whereas $\pi_{n, \delta + 2}(\mathcal{Q}')$ contains the upper facets of $\pi_{n, \delta + 2}(\Gamma)$. Here we say that $\mathcal{Q}'$ is an \emph{increasing flip} of $\mathcal{Q}$.

The cyclic zonotope $Z(n,\delta+1)$ possesses two canonical cubillages, one given by the subcomplex $\mathcal{Q}_{l}$ of $Z(n, n)$ such that $\pi_{n, \delta + 2}(\mathcal{Q}_{l})$ consists of the lower facets of $Z(n,\delta+2)$, which we call the \emph{lower cubillage}, and the other given by the subcomplex $\mathcal{Q}_{u}$ of $Z(n, n)$ such that $\pi_{n, \delta + 2}(\mathcal{Q}_{u})$ consists of the upper facets of $Z(n,\delta+2)$, which we call the \emph{upper cubillage}. The lower cubillage of $Z(n,\delta+1)$ gives the unique minimum of the poset $\mathcal{B}(n,\delta+1)$, and the upper cubillage gives the unique maximum.

\subsubsection{Separated collections}

We now explain how one may characterise cubillages as separated collections of subsets, as shown in \cite{gp}.

The subsets $E \subseteq [n]$ are naturally identified with the corresponding points $\xi_{E}:=\sum_{e \in E}\xi_{e}$ in $Z(n, n)$, where $\xi_{\emptyset} := \mathbf{0}$. This represents each vertex of a cubillage $\mathcal{Q}$ as a subset of $[n]$. For a cubillage $\mathcal{Q}$ of $Z(n,\delta+1)$, the collection of subsets corresponding to its vertices is called the \emph{spectrum} of $\mathcal{Q}$ and is denoted by $\mathrm{Sp}(\mathcal{Q})$. Each cube in $\mathcal{Q}$ is viewed as the Minkowski sum of line segments \[\overline{\xi_{E}\xi_{E \cup \{a_{i}\}}}\] for some set $A$ with $\# A = \delta + 1$ and $E\subseteq [n]\setminus A $. Here we call $\xi_{E}$ the \emph{initial vertex} of the cube, $\xi_{E \cup A}$ the \emph{final vertex}, and $A$ the set of \emph{generating vectors}.

We say that, given two sets $A,B \subseteq [n]$, $A$ \emph{$\delta$-interweaves} $B$ if there exist $i_{\delta+1}, i_{\delta-1}, \ldots \in B \setminus A$ and $i_{\delta}, i_{\delta-2}, \ldots \in A \setminus B$ such that \[i_{0}<i_{1}<\dots< i_{\delta+1}.\] We also say that $\{i_{\delta+1},i_{\delta-1}, \dots\}$ and $\{i_{\delta},i_{\delta-2}, \dots\}$ \emph{witness} that $A$ $\delta$-interweaves $B$. If either $A$ $\delta$-interweaves $B$ or $B$ $\delta$-interweaves $A$, then we say that $A$ and $B$ are \emph{$\delta$-interweaving}. If $A$ $\delta$-interweaves $B$ as above and $B \setminus A = \{i_{\delta+1}, i_{\delta-1}, \ldots\}$ and $A \setminus B = \{i_{\delta}, i_{\delta-2}, \ldots\}$, then we say that \emph{$A$ tightly $\delta$-interweaves $B$}, in the manner of \cite{bbge}. If $A$ and $B$ are not $\delta$-interweaving then we say that $A$ and $B$ are \emph{$\delta$-separated}, following \cite{gp,dkk}. We call a collection $\mathcal{C}\subseteq 2^{[n]}$ $\delta$-separated if it is pairwise $\delta$-separated.

If $\delta = 2d$, then being $\delta$-interweaving is the same as being $(d + 1)$-interlacing in the terminology of \cite{bbge} and $(d + 1)$-intertwining in the terminology of \cite{njw-jm}. We choose new terminology because we wish to have an opposite of $\delta$-separated for $\delta$ odd as well as $\delta$ even. 

It follows from \cite[Theorem 2.7]{gp} that the correspondence $\mathcal{Q} \mapsto \mathrm{Sp}(\mathcal{Q})$ gives a bijection between the set of cubillages on $Z(n,\delta+1)$ and the set of $\delta$-separated collections of maximal size in $2^{[n]}$. In particular, for any cubillage $\mathcal{Q}$ of $Z(n,\delta+1)$, we have that $\# \mathrm{Sp}(\mathcal{Q}) = \Sigma_{i=0}^{\delta+1}\binom{n}{i}$, which is the maximal size of a $\delta$-separated collection in $2^{[n]}$.

For $A \subseteq [n]$, if $\pi_{n, \delta + 1}(\xi_{A})$ is a boundary vertex of the zonotope $Z(n,\delta+1)$, then $\xi_{A}$ is a vertex of every cubillage of $Z(n, \delta + 1)$, and hence $A$ is in every $\delta$-separated collection in $2^{[n]}$ of maximal size. Moreover, the subsets $A \subseteq [n]$ such that $\pi_{n, \delta + 1}(\xi_{A})$ is a boundary vertex of the zonotope $Z(n,\delta+1)$ are precisely those subsets which are $\delta$-separated from every other subset of $[n]$. Hence the subsets of interest are those which project to the interior of the zonotope $Z(n,\delta+1)$. The vertices of the zonotope $Z(n,\delta+1)$ are known to be in bijection with the number of regions of the arrangement of $(\delta-1)$-spheres associated with the set of points $\Xi = \{\xi_{1}, \dots, \xi_{n}\}$ on the Veronese curve, see \cite[Proposition 2.2.2]{blswz}. Since no set of $\delta$ points of $\Xi$ lie in a linear hyperplane, the number of regions of this arrangement of $(\delta - 1)$-spheres is the maximal number of \[\binom{n-1}{\delta} + \sum_{i=0}^{\delta}\binom{n}{i}.\] (For instance, see \cite[Problem 4, p.73]{comtet}.) Hence a cubillage $\mathcal{Q}$ of $Z(n,\delta+1)$ has \[\sum_{i=0}^{\delta+1}\binom{n}{i}-\left(\binom{n-1}{\delta} + \sum_{i=0}^{\delta}\binom{n}{i}\right)=\binom{n-1}{\delta+1}\] vertices which project to the interior of $Z(n,\delta+1)$ if $n > \delta + 1$, and $0$ otherwise. We call a point $\xi_{A} \in \mathbb{R}^{n}$ an \emph{internal point in $Z(n, \delta + 1)$} if $\pi_{n, \delta + 1}(\xi_{A})$ lies in the interior of $Z(n, \delta + 1)$. We call a vertex $\xi_{A}$ of a cubillage $\mathcal{Q} \subset \mathbb{R}^{n}$ of $Z(n, \delta + 1)$ \emph{internal} if $\xi_{A}$ is an internal point in $Z(n,\delta + 1)$. Given a cubillage $\mathcal{Q}$ of $Z(n, \delta + 1)$, we define its \emph{internal spectrum} $\intsp(\mathcal{Q})$ to consist of the elements of $\mathrm{Sp}(\mathcal{Q})$ which correspond to internal vertices of $\mathcal{Q}$.

By \cite[(2.7)]{dkk}, $\xi_{A}$ is an internal point in $Z(n,\delta+1)$ if and only if
\begin{itemize}
\item $\delta = 2d$ and $A$ is a cyclic $l$-ple interval for $l \geqslant d + 1$, or
\item $\delta = 2d + 1$ and $A$ is an $l$-ple interval for $l \geqslant d + 2$, or a $(d + 1)$-ple interval containing neither $1$ nor $n$.
\end{itemize}

We will also need the following concepts from \cite{dkk-interrelations}. Given a cubillage $\mathcal{Q}$ of $Z(n, \delta +1)$ and a subcomplex $\mathcal{M}$ of $\mathcal{Q}$, we say that $\mathcal{M}$ is a \emph{membrane} in $\mathcal{Q}$ if $\mathcal{M}$ is a cubillage of $Z(n, \delta)$. We say that an edge in a cubillage $\mathcal{Q}$ from $\xi_{E}$ to $\xi_{E \cup \{i\}}$ is an edge \emph{of colour $i$}, where $E \subseteq [n] \setminus \{i\}$ is any subset. For a cubillage $\mathcal{Q}$ of $Z(n, \delta + 1)$ and $i \in [n]$, we define the \emph{$i$-pie} $\Pi_{i}(\mathcal{Q})$ to be the subcomplex of $\mathcal{Q}$ given by all the cubes which have an edge of colour $i$. In \cite[Chapter 7]{ziegler}, the $i$-pie is called the \emph{$i$-th zone}.

By \cite{dkk,gp}, we can obtain a cubillage $\mathcal{Q}/i$ from $\mathcal{Q}$ by contracting the edges of colour $i$ until they have length zero. The cubillage $\mathcal{Q}/i$ is known as the $i$-contraction of $\mathcal{Q}$. The image of the $n$-pie $\Pi_{n}(\mathcal{Q})$ is a membrane in $\mathcal{Q}/n$, but this is not in general true for $1 < i < n$, by \cite[(4.4)]{dkk}. An example of $4$-contraction is shown in Figure~\ref{fig:ipie}. Here the 4-pie is shown in red on the left-hand cubillage, and this is contracted to zero in the right-hand cubillage, where its image is a membrane. Note that here we are illustrating cubillages of $Z(4, 2)$ and $Z(3, 2)$ by their images under the projection maps $\pi_{4, 2}$ and $\pi_{3, 2}$ respectively. We will always illustrate cubillages in this way.

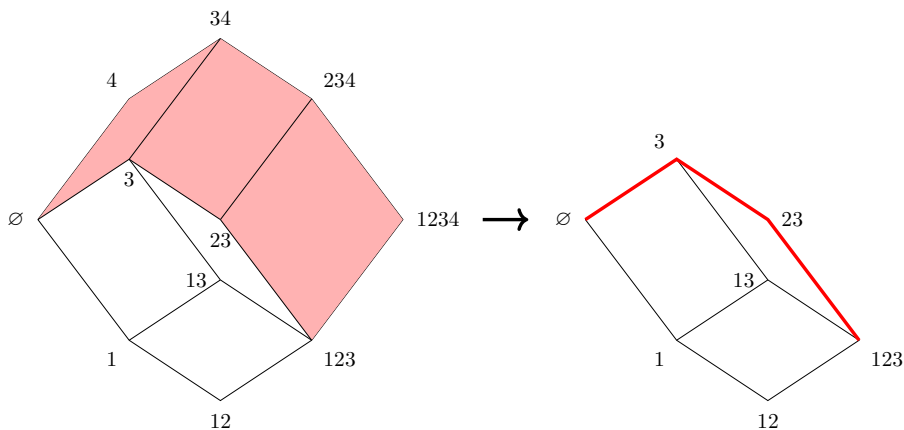
\begin{figure}
\caption{$4$-contraction}\label{fig:ipie}
\[
\scalebox{0.8}{
\begin{tikzpicture}

\begin{scope}[xscale=0.75]

% Boundary vertices

\coordinate(0) at (0,0);
\node at (0)[left = 1mm of 0]{$\emptyset$};
\coordinate(1) at (2,-2);
\node at (1)[below left = 1mm of 1]{1};
\coordinate(12) at (4,-3);
\node at (12)[below = 1mm of 12]{12};
\coordinate(123) at (6,-2);
\node at (123)[below right = 1mm of 123]{123};
\coordinate(1234) at (8,0);
\node at (1234)[right = 1mm of 1234]{1234};
\coordinate(234) at (6,2);
\node at (234)[above right = 1mm of 234]{234};
\coordinate(34) at (4,3);
\node at (34)[above = 1mm of 34]{34};
\coordinate(4) at (2,2);
\node at (4)[above left = 1mm of 4]{4};

% Boundary edges

\draw (0) -- (1) -- (12) -- (123) -- (1234) -- (234) -- (34) -- (4) -- (0);

% Internal vertices

\coordinate(3) at (2,1);
\coordinate(13) at (4,-1);
\coordinate(23) at (4,0);

% Filled cubes

\draw[fill=red!30,draw=none] (0) -- (4) -- (34) -- (3) -- (0);
\draw[fill=red!30,draw=none] (3) -- (34) -- (234) -- (23) -- (3);
\draw[fill=red!30,draw=none] (23) -- (234) -- (1234) -- (123) -- (23);

% Internal edges

\draw (0) -- (3);
\draw (3) -- (34);
\draw (3) -- (23);
\draw (3) -- (13);
\draw (1) -- (13);
\draw (13) -- (123);
\draw (23) -- (123);
\draw (23) -- (234);

% Labels of internal vertices

\node at (3) [below = 1mm of 3]{3};
\node at (13) [left = 1mm of 13]{13};
\node at (23) [below = 1mm of 23]{23};

%%%%%%%%%%%%%%%%%%%%%%%%%%%%%%%%%%%%%%%%%%%%
% CONTRACTED VERSION
%%%%%%%%%%%%%%%%%%%%%%%%%%%%%%%%%%%%%%%%%%%%

% Arrow from left to right

\draw[->,ultra thick] (9.75, 0) -- (10.75, 0);

% Boundary vertices

\coordinate(0) at (12,0);
\node at (0)[left = 1mm of 0]{$\emptyset$};
\coordinate(1) at (14,-2);
\node at (1)[below left = 1mm of 1]{1};
\coordinate(12) at (16,-3);
\node at (12)[below = 1mm of 12]{12};
\coordinate(123) at (18,-2);
\node at (123)[below right = 1mm of 123]{123};
%\coordinate(1234) at (20,0);
%\node at (1234)[right = 1mm of 1234]{1234};
%\coordinate(234) at (18,2);
%\node at (234)[above right = 1mm of 234]{234};
%\coordinate(34) at (16,3);
%\node at (34)[above = 1mm of 34]{34};
%\coordinate(4) at (14,2);
%\node at (4)[above left = 1mm of 4]{4};

% Internal vertices

\coordinate(3) at (14,1);
\coordinate(13) at (16,-1);
\coordinate(23) at (16,0);

% Boundary edges

\draw (0) -- (1) -- (12) -- (123) -- (23) -- (3) -- (0);

%(1234) -- (234) -- (34) -- (4) -- (0);

% Internal edges

%\draw (0) -- (3);
%\draw (3) -- (34);
%\draw (3) -- (23);
\draw (3) -- (13);
\draw (1) -- (13);
\draw (13) -- (123);
%\draw (23) -- (123);
%\draw (23) -- (234);

% Labels of internal vertices

\node at (3) [above left = 1mm of 3]{3};
\node at (13) [left = 1mm of 13]{13};
\node at (23) [right = 1mm of 23]{23};

% Membrane

\draw[red, ultra thick] (0) -- (3) -- (23) -- (123);

\end{scope}

\end{tikzpicture}
}
\]
\end{figure}

\subsubsection{Admissible orders}\label{sect:admissible_orders}

The original definition of the higher Bruhat orders from \cite{ms} is as follows. Given $A \in \binom{[n]}{\delta+2}$, the set \[P(A)=\left\lbrace B \mathrel{\Big|} B \in \binom{[n]}{\delta+1}, B \subset A \right\rbrace \] is called the \emph{packet} of $A$. The set $P(A)$ is naturally ordered by the \emph{lexicographic order}, where $P(A)\setminus a_{i} < P(A) \setminus a_{j}$ if and only if $j < i$.

An ordering $\alpha$ of $\binom{[n]}{\delta+1}$ is \emph{admissible} if the elements of any packet appear in either lexicographic or reverse-lexicographic order under $\alpha$. Two orderings $\alpha$ and $\alpha'$ of $\binom{[n]}{\delta + 1}$ are \emph{equivalent} if they differ by a sequence of interchanges of pairs of adjacent elements that do not lie in a common packet. Note that these interchanges preserve admissibility. We use $[\alpha]$ to denote the equivalence class of $\alpha$.

The \emph{inversion set} $\mathrm{inv}(\alpha)$ of an admissible order $\alpha$ is the set of all $(\delta+2)$-subsets of $[n]$ whose packets appear in reverse-lexicographic order in $\alpha$. Note that inversion sets are well-defined on equivalence classes of admissible orders.

The higher Bruhat poset $\mathcal{B}(n,\delta+1)$ is the partial order on equivalence classes of admissible orders of $\binom{[n]}{\delta+1}$ where $[\alpha] \lessdot [\alpha']$ if $\mathrm{inv}(\alpha')=\mathrm{inv}(\alpha) \cup \{ A\}$ for $A \in \binom{[n]}{\delta+2}\setminus\mathrm{inv}(\alpha)$.

One can explain the bijection between cubillages of $Z(n,\delta+1)$ and admissible orders on $\binom{[n]}{\delta+1}$. Let $\mathcal{Q}$ be a cubillage of $Z(n,\delta+1)$ corresponding to an equivalence class $[\alpha]$ of admissible orders on $\binom{[n]}{\delta+1}$. It follows from \cite{thomas-bst} that the cubes of $\mathcal{Q}$ are in bijection with the elements of $\binom{[n]}{\delta+1}$ via sending a cube to its set of generating vectors. A packet which can be inverted corresponds to a set of lower facets of $\pi_{n, \delta + 2}(\Gamma)$, where $\Gamma$ is a $(\delta+2)$-face $\Gamma$ of $Z(n, n)$. Inverting the packet corresponds to an increasing flip: exchanging the lower facets of $\pi_{n, \delta + 2}(\Gamma)$ for its upper facets.

Hence, a cubillage $\mathcal{Q}$ of $Z(n,\delta+1)$ is determined once, for every element of $\binom{[n]}{\delta + 1}$, one knows the initial vertex of the cube with that set of generating vectors. Let $\alpha$ be an admissible order of $\binom{[n]}{\delta+1}$ corresponding to a cubillage $\mathcal{Q}$ of $Z(n,\delta+1)$ and let $\Delta$ be the cube of $\mathcal{Q}$ with set of generating vectors $I$ and initial vertex $\xi_{E}$. Then, given $e \in [n]\setminus I$, we have that $e \in E$ if and only if either
\begin{itemize}
\item $I \cup \{e\} \notin \mathrm{inv}(\alpha)$ and $e$ is an odd gap in $I$, or
\item $I \cup \{e\} \in \mathrm{inv}(\alpha)$ and $e$ is an even gap in $I$.
\end{itemize}
This follows from \cite[Theorem 2.1]{thomas-bst} if one swaps the sign convention for $\delta + 1$ odd. This makes the statement simpler and reveals connections with the paper \cite{dm-h}, as we explain in Section~\ref{sect:vis}. An analogous statement was shown for more general zonotopes in \cite[Lemma 5.13]{gpw}.

Conversely, given a cubillage $\mathcal{Q}$ of $Z(n, \delta + 1)$, one can determine an equivalence class of admissible orders of $\binom{[n]}{\delta + 1}$. Define a partial order on the cubes of the cubillage $\mathcal{Q}$ by $\Delta \lessdot \Delta'$ if $\pi_{n, 2d + 1}(\Delta) \cap \pi_{n, 2d + 1}(\Delta')$ is an upper facet of $\pi_{n, 2d + 1}(\Delta)$ and a lower facet of $\pi_{n, 2d + 1}(\Delta')$. The linear extensions of this partial order then comprise the admissible orders in the equivalence class $[\alpha]$ corresponding to $\mathcal{Q}$, by \cite[Lemma 2.2]{ziegler-bruhat} and \cite{ms,thomas-bst}.

\subsection{Higher Stasheff--Tamari orders}\label{sect-hst}

In this section we give the definition of the first higher Stasheff--Tamari orders. These were originally defined by Kapranov and Voevodsky under the name the \emph{higher Stasheff orders} in the context of higher category theory \cite[Definition 3.3]{kv-poly}. This was built upon by Edelman and Reiner, who introduced the \emph{first} and \emph{second higher Stasheff--Tamari orders} in \cite{er}. Thomas later proved that the first higher Stasheff--Tamari orders were the same as the higher Stasheff orders of Kapranov and Voevodsky \cite[Proposition 3.3]{thomas-bst}. The definition of the first higher Stasheff--Tamari orders is similar in style to the geometric definition of the higher Bruhat orders using cubillages.

The \emph{moment curve} is defined by $p_{t}=(t, t^{2}, \dots , t^{\delta}) \subseteq \mathbb{R}^{\delta}$ for $t \in \mathbb{R}$. Choose $t_{1}, \dots , t_{n} \in \mathbb{R}$ such that $t_{1} < t_{2} < \dots < t_{n}$ and $n \geqslant \delta + 1$. The \emph{cyclic polytope} $C(n, \delta)$ is defined to be the convex polytope $\mathrm{conv}(p_{t_{1}}, \dots , p_{t_{n}})$. The properties of the cyclic polytope do not depend on the exact choice of $\{t_{1}, \dots, t_{n}\} \subset \mathbb{R}$. Hence, for ease we set $t_{i}=i$.

A \emph{triangulation} of the cyclic polytope $C(n, \delta)$ is a subcomplex $\mathcal{T}$ of $C(n, n - 1)$ such that $\pi_{n-1, \delta} \colon \mathcal{T} \to C(n, \delta)$ is a bijection. After \cite{kv-poly, thomas-bst}, we define the \emph{first higher Stasheff--Tamari poset} $\mathcal{S}(n,\delta)$ as follows. The elements of $\mathcal{S}(n,\delta)$ are triangulations of $C(n,\delta)$. The covering relations of $\mathcal{S}(n,\delta)$ are given by pairs of triangulations $\mathcal{T} \lessdot \mathcal{T}'$ where there is a $(\delta + 1)$-face $\Sigma$ of $C(n, n - 1)$ such that $\mathcal{T}\setminus \Sigma = \mathcal{T}' \setminus \Sigma$ and $\pi_{n - 1, \delta + 1}(\mathcal{T})$ contains the lower facets of $\pi_{n - 1, \delta + 1}(\Sigma)$, whereas $\pi_{n - 1, \delta + 1}(\mathcal{T}')$ contains the upper facets of $\pi_{n - 1, \delta + 1}(\Sigma)$. Here we say that $\mathcal{T}'$ is an \emph{increasing flip} of $\mathcal{T}$.

The cyclic polytope $C(n,\delta)$ possesses two canonical triangulations, one given by the subcomplex $\mathcal{T}_{l}$ of $C(n, n - 1)$ such that $\pi_{n - 1, \delta + 1}(\mathcal{T}_{l})$ consists of the lower facets of $C(n, \delta+1)$, known as the \emph{lower triangulation}, and the other given by the subcomplex $\mathcal{T}_{u}$ of $C(n, n - 1)$ such that $\pi_{n - 1, \delta + 1}(\mathcal{T}')$ consists of the upper facets of $C(n, \delta+1)$, known as the \emph{upper triangulation}. The lower triangulation of $C(n,\delta)$ gives the unique minimum of the poset $\mathcal{S}(n,\delta)$ and the upper triangulation gives the unique maximum.

Given a subset $A \subseteq [n]$ with $\# A = k + 1$, we write $|A| := \mathrm{conv}(p_{a_{0}}, \dots, p_{a_{k}})$ for its geometric realisation as a simplex in $\mathbb{R}^{n - 1}$. One may combinatorially describe the lower facets and upper facets of $C(n, \delta)$, and hence the lower and upper triangulations of $C(n, \delta - 1)$. Gale's Evenness Criterion \cite[Theorem 3]{gale}\cite[Lemma 2.3]{er} states that, for $F \subseteq [n]$ with $\# F = \delta$, we have that $\pi_{n - 1, \delta}|F|$ is an upper facet of $C(n, \delta)$ if and only if $F$ is an odd subset, and that $\pi_{n - 1, \delta}|F|$ is a lower face of $C(n, \delta)$ if and only if $F$ is an even subset. We remove excess brackets, so that here $\pi_{n - 1, \delta}|F| = \pi_{n - 1, \delta}(|F|)$. 

We call a $\floor{\delta/2}$-simplex $|A| \subset \mathbb{R}^{n - 1}$ \emph{internal in} $C(n, \delta)$ if $\pi_{n - 1, \delta}|A|$ does not lie within a facet of $C(n,\delta)$. A $\floor{\delta/2}$-simplex $|A|$ of a triangulation $\mathcal{T} \subset \mathbb{R}^{n - 1}$ of $C(n,\delta)$ is an \emph{internal $\floor{\delta/2}$-simplex} if it is internal in $C(n, \delta)$. It is clear that a triangulation of a convex polygon is determined by the arcs of the triangulation; similarly, a triangulation of $C(n,\delta)$ is determined by the internal $\floor{\delta/2}$-simplices of the triangulation, by a theorem of Dey \cite{dey}. Hence, for a triangulation $\mathcal{T}$ of $C(n,\delta)$ we denote by \[\simp(\mathcal{T}) := \left\lbrace A \in \binom{[n]}{\floor{\delta/2} + 1} \mathrel{\Big|} |A| \text{ is an internal }\floor{\delta/2}\text{-simplex of }\mathcal{T}\right\rbrace.\]

By \cite[Lemma 2.1]{ot} and \cite[Lemma 4.2]{njw-hst}, given $A \in \binom{[n]}{\floor{\delta/2} + 1}$, we have that $|A|$ is an internal $\floor{\delta/2}$-simplex in $C(n,\delta)$ if
\begin{itemize}
\item $\delta = 2d$ and $A$ is a cyclic $(d + 1)$-ple interval, or
\item $\delta = 2d + 1$ and $A$ is a $(d + 1)$-ple interval containing neither $1$ nor $n$.
\end{itemize}

\begin{observation}\label{obs:internal}
Given $A \in \binom{[n]}{\floor{\delta/2} + 1}$, we have that $|A|$ is an internal $\floor{\delta/2}$-simplex in $C(n, \delta)$ if and only if $\xi_{A}$ is an internal point in $Z(n, \delta + 1)$.
\end{observation}

A \emph{circuit} of a cyclic polytope $C(n,\delta)$ is a pair of disjoint subsets $X, Y \subseteq [n]$ which are inclusion-minimal with the property that $\pi_{n - 1, \delta}|X| \cap \pi_{n - 1, \delta}|Y| \neq \emptyset$. If $A, B \subseteq [n]$ are such that $A \supseteq X$ and $B \supseteq Y$ where $(X, Y)$ is a circuit of $C(n, \delta)$, then we say that $\pi_{n - 1, \delta}|A|$ and $\pi_{n - 1, \delta}|B|$ intersect \emph{transversely} in $C(n, \delta)$. By \cite{breen}, the circuits of $C(n,\delta)$ are the pairs $(X, Y)$ and $(Y, X)$ such that $\#X = \floor{\delta/2}+1$, $\#Y = \ceil{\delta/2}+1$, and $X$ $\delta$-interweaves $Y$. This also follows from the description of the oriented matroid given by a cyclic polytope \cite{b-lv,sturmfels,cd}. We will later use the fact that if $|A|$ and $|B|$ are simplices in the same triangulation, then there is no circuit $(X, Y)$ such that $A \supseteq X$ and $B \supseteq Y$ \cite[Proposition 2.2]{rambau}.

\section{Interpretations}\label{sect:g_interpretations}

In this section we study the map \[g\colon\mathcal{B}(n,\delta+1) \rightarrow \mathcal{S}(n,\delta).\] We give three different interpretations of this map, corresponding to the three different ways of defining the higher Bruhat orders.

\subsection{Cubillages}\label{sect:g_geom}

Here we give our principal definition of the map $g$. This definition is geometric and uses the interpretation of $\mathcal{B}(n,\delta+1)$ in terms of cubillages. This was how the map was considered in \cite[Appendix B]{dkk-survey}, where Lemma~\ref{lem:vertexFigTriang} and Proposition~\ref{prop:covRel} were both noted.

\begin{lemma}\label{lem:vertexFigTriang}
If $\mathcal{Q}$ is a cubillage of $Z(n,\delta+1)$, then the vertex figure of $\mathcal{Q}$ at $\xi_{\emptyset}$ gives a triangulation of $C(n, \delta)$.
\end{lemma}
\begin{proof}
Let $H_{k}$ denote the affine hyperplane \[H_{k}:=\{(x_{1}, \dots, x_{k}) \in \mathbb{R}^{k} \mid x_{1}=1\}.\] The vertex figure of the zonotope $Z(n,k)$ at the vertex $\xi_{\emptyset}$ can be given by the intersection $Z(n, k) \cap H_{k}$. It is clear from the definitions of $Z(n, k)$ and $C(n, k)$ that this intersection is the cyclic polytope $C(n,k)$. The vertex figure of the cubillage $\mathcal{Q}$ of $Z(n,\delta+1)$ at $\xi_{\emptyset}$ then induces a subcomplex $\mathcal{T} = \mathcal{Q} \cap H_{n}$ of $C(n, n - 1)$. This subcomplex $\mathcal{T}$ is a triangulation of $C(n, \delta)$ because we have that $\pi_{n, \delta + 1} \colon \mathcal{Q} \to Z(n, \delta + 1)$ is a bijection, which then restricts to a bijection from $\mathcal{Q} \cap H_{n} = \mathcal{T}$ to $Z(n, \delta + 1) \cap H_{\delta + 1} = C(n, \delta)$.
\end{proof}

Hence we define the map
\begin{align*}
g\colon \mathcal{B}(n,\delta+1)&\rightarrow\mathcal{S}(n,\delta) \\
\mathcal{Q} &\mapsto \mathcal{Q} \cap H_{n}.
\end{align*}
For the purposes of this paper, this is the definition of the map $g$, and the characterisations in Section~\ref{sect:g_comb} and Section~\ref{sect:vis} are simply other interpretations.

\begin{remark}
The intersections of cubillages with the hyperplanes given by $x_{1} = l$ for $l \in [n - 1]$ have been the objects of significant study in the literature. For three-dimensional zonotopes, such cross-sections are dual to plabic graphs \cite{galashin}, which arise in the combinatorics associated to Grassmannians \cite{post-grass,post_icm}. When the cubillage is \emph{regular}, such graphs arise in the study of KP solitons \cite{huang_thesis,kk,gpw}, and it is this connection that lies behind the definition of the higher Tamari orders in \cite{dm-h}. The paper \cite{olarte_santos} studies \emph{hypersimplicial subdivisions} and shows that, in general, only a subset of these come from cross-sections of subdivisions of zonotopes. This means that the analogues of the map $g$ for other cross-sections of cubillages are not generally surjective. In \cite{dkk-interrelations,dkk-weak,dkk-symmetric}, rather than studying the intersection of a cubillage with these hyperplanes, the \emph{fragmentation} of a cubillage into different pieces cut by these hyperplanes is studied.
\end{remark}

We identify the hyperplane $H_{n}$ with the space $\mathbb{R}^{n - 1}$, so that we can consider $C(n, n - 1)$ sitting inside it as usual. In particular, we abuse notation by using $\pi_{n - 1, \delta + 1}$ to denote the restriction $\pi_{n, \delta + 2}|_{H_{n}}\colon H_{n} \to H_{\delta + 2}$. This convention is illustrated in the following proof, in which we examine how $g$ interacts with increasing flips.

\begin{proposition}\label{prop:covRel}
If $\mathcal{Q}, \mathcal{Q}'$ are cubillages of $Z(n, \delta + 1)$ such that $\mathcal{Q} \lessdot \mathcal{Q}'$, then either $g(\mathcal{Q}) = g(\mathcal{Q}')$ or $g(\mathcal{Q}) \lessdot g(\mathcal{Q}')$.
\end{proposition}
\begin{proof}
Let $\mathcal{Q}$ and $\mathcal{Q}'$ be two cubillages such that $\mathcal{Q} \lessdot \mathcal{Q}'$. Let $\Gamma$ be the $(\delta + 2)$-face of $Z(n, n)$ which induces the increasing flip, and let the initial vertex of $\Gamma$ be $\xi_{E} = (x_{1}, \dots, x_{n})$. Then $\mathcal{Q}$ and $\mathcal{Q}'$ differ only in that $\pi_{n, \delta + 2}(\mathcal{Q})$ contains the lower facets of $\pi_{n, \delta + 2}(\Gamma)$ and $\pi_{n, \delta + 2}(\mathcal{Q}')$ contains the upper facets of $\pi_{n, \delta + 2}(\Gamma)$.

The intersection $\Gamma \cap H_{n}$ consists of more than a single point if and only if $E = \emptyset$. This is because, given $(y_{1}, \dots, y_{n}) \in \Gamma$, we have $y_{1} \geqslant x_{1} = \# E$. Hence if $\# E > 1$, then $\Gamma \cap H_{n} = \emptyset$; and if $\# E = 1$, then $\Gamma \cap H_{n} = \xi_{E}$. Thus if $E \neq \emptyset$, then $\mathcal{Q}$ and $\mathcal{Q}'$ both have the same intersection with the hyperplane $H_{n}$, so that $g(\mathcal{Q})=g(\mathcal{Q}')$.

If $E = \emptyset$, then $\pi_{n, \delta + 2}(\Gamma) \cap H_{\delta + 2}$ is the $(\delta+1)$-simplex $\pi_{n - 1, \delta + 1}|A|$, where $A$ is the generating set of $\Gamma$. We then have that $g(\mathcal{Q})$ and $g(\mathcal{Q}')$ differ only in that $\pi_{n - 1, \delta + 1}(g(\mathcal{Q}))$ contains the lower facets of $\pi_{n - 1, \delta + 1}|A|$, whereas $\pi_{n - 1, \delta + 1}(g(\mathcal{Q}'))$ contains the upper facets of $\pi_{n - 1, \delta + 1}|A|$. Hence $g(\mathcal{Q}) \lessdot g(\mathcal{Q}')$.
\end{proof}

Recall that if $(X, \leqslant)$ and $(Y, \leqslant)$ are posets, and $f \colon X \to Y$ is map such that we have $f(x) \leqslant f(x')$ whenever $x \leqslant x'$, then $f$ is called \emph{order-preserving}.

\begin{corollary}\label{cor-ord-pres}
The map $g \colon \mathcal{B}(n, \delta + 1) \to \mathcal{S}(n, \delta)$ is order-preserving.
\end{corollary}

\begin{example}\label{ex:vertex_figures}
We now give two examples of taking the vertex figure of a cubillage of $Z(n,\delta+1)$ at $\xi_{\emptyset}$.

First, consider the cubillage $\mathcal{Q}_{1}$ of $Z(4,2)$ shown in Figure~\ref{fig:q1}. As we did above, we can find the vertex figure of $\mathcal{Q}_{1}$ at $\xi_{\emptyset}$ by intersecting with the hyperplane $H_{4}$, as shown. We thus obtain the triangulation $g(\mathcal{Q}_{1})=\mathcal{T}_{1}$ of $C(4,1)$ shown in Figure~\ref{fig:t1}.

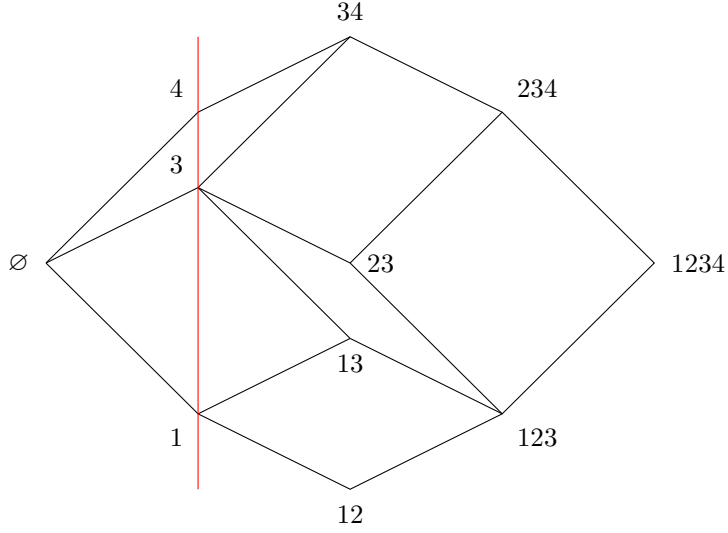
\begin{figure}
\caption{The cubillage $\mathcal{Q}_{1}$ of $Z(4,2)$ intersected with $H_{4}$.}\label{fig:q1}
\[
\begin{tikzpicture}

% Boundary vertices

\coordinate(0) at (0,0);
\node at (0)[left = 1mm of 0]{$\emptyset$};
\coordinate(1) at (2,-2);
\node at (1)[below left = 1mm of 1]{1};
\coordinate(12) at (4,-3);
\node at (12)[below = 1mm of 12]{12};
\coordinate(123) at (6,-2);
\node at (123)[below right = 1mm of 123]{123};
\coordinate(1234) at (8,0);
\node at (1234)[right = 1mm of 1234]{1234};
\coordinate(234) at (6,2);
\node at (234)[above right = 1mm of 234]{234};
\coordinate(34) at (4,3);
\node at (34)[above = 1mm of 34]{34};
\coordinate(4) at (2,2);
\node at (4)[above left = 1mm of 4]{4};

% Boundary edges

\draw (0) -- (1) -- (12) -- (123) -- (1234) -- (234) -- (34) -- (4) -- (0);

% Internal vertices

\coordinate(3) at (2,1);
\coordinate(13) at (4,-1);
\coordinate(23) at (4,0);

% Labels of internal vertices

\node at (3) [above left = 1mm of 3]{3};
\node at (13) [below = 1mm of 13]{13};
\node at (23) [right = 1mm of 23]{23};

% Internal edges

\draw (0) -- (3);
\draw (3) -- (34);
\draw (3) -- (23);
\draw (3) -- (13);
\draw (1) -- (13);
\draw (13) -- (123);
\draw (23) -- (123);
\draw (23) -- (234);

% Hyperplane

\draw[red] (2,-3) -- (2,3);

\end{tikzpicture}
\]
\end{figure}

\begin{figure}
\caption{The triangulation $g(\mathcal{Q}_{1})=\mathcal{T}_{1}$ of $C(4,1)$.}\label{fig:t1}
\[
\begin{tikzpicture}

\draw[red] (1,0) -- (4,0);

\node(1) at (1,0) {$\bullet$};
\node at (1)[below = 1mm of 1]{1};
\node(3) at (3,0) {$\bullet$};
\node at (3)[below = 1mm of 3]{3};
\node(4) at (4,0) {$\bullet$};
\node at (4)[below = 1mm of 4]{4};
\end{tikzpicture}
\]
\end{figure}
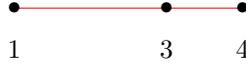

Secondly, consider the cubillage $\mathcal{Q}_{2}$ of $Z(4,3)$ illustrated in Figure~\ref{fig:q2}. This cubillage possesses four cubes, two of which share the face highlighted in blue. The hyperplane $H_{4}$ is shown here in red. The intersection gives the triangulation $g(\mathcal{Q}_{2})=\mathcal{T}_{2}$ of $C(4,2)$ shown in Figure~\ref{fig:t2}.

\begin{figure}
\caption{The cubillage $\mathcal{Q}_{2}$ of $Z(4,3)$.}\label{fig:q2}
\[
\begin{tikzpicture}

\begin{scope}[xscale=0.9]

% Vertices

\coordinate(0) at (0,-5);
\coordinate(1) at (-3,-2);
\coordinate(2) at (-3,-3);
\coordinate(3) at (3,-3);
\coordinate(4) at (3,-2);
\coordinate(12) at (-6,0);
\coordinate(23) at (0,-1);
\coordinate(13) at (0,0);
\coordinate(14) at (0,1);
\coordinate(34) at (6,0);
\coordinate(123) at (-3,2);
\coordinate(124) at (-3,3);
\coordinate(234) at (3,2);
\coordinate(134) at (3,3);
\coordinate(1234) at (0,5);

% Back edges

\draw[gray] (0) -- (4);
\draw[gray] (4) -- (34);
\draw[gray] (1) -- (14);
\draw[gray] (4) -- (14);
\draw[gray] (14) -- (124);
\draw[gray] (14) -- (134);
\draw[gray] (0) -- (1);
\draw[gray] (1) -- (12);

% Shaded section

\draw[fill=blue!30,draw=none] (0) -- (3) -- (13) -- (1) -- (0);
\draw[fill=red!30,draw=none] (1) -- (2) -- (3) -- (1);

\path[name path = line 1] (1) -- (4);
\path[name path = line 2] (3) -- (13);

\path [name intersections={of = line 1 and line 2}];
\coordinate (a) at (intersection-1);

\draw[fill=red!30,draw=none] (a) -- (4) -- (3) -- (a);

% Middle edges

\draw[dotted,red] (1) -- (2) -- (3) -- (4) -- (1);
\draw[dotted,blue] (1) -- (3);
\draw[dashed] (1) -- (13);
\draw[dashed] (3) -- (13);
\draw[dashed] (13) -- (123);

% Front edges

\draw (0) -- (2);
\draw (0) -- (3);
\draw (2) -- (12);
\draw (2) -- (23);
\draw (3) -- (23);
\draw (12) -- (123);
\draw (23) -- (123);
\draw (3) -- (34);
\draw (23) -- (234);
\draw (12) -- (124);
\draw (34) -- (134);
\draw (34) -- (234);
\draw (123) -- (1234);
\draw (124) -- (1234);
\draw (234) -- (1234);
\draw (134) -- (1234);

% Other internal edges

\draw[dashed] (13) -- (134);

% Vertex labels

\node at (0) [below = 1mm of 0]{$\emptyset$};
\node at (1) [above = 1mm of 1]{1};
\node at (2) [below left = 1mm of 2]{2};
\node at (3) [below right = 1mm of 3]{3};
\node at (4) [above = 1mm of 4]{4};
\node at (12) [left = 1mm of 12]{12};
\node at (23) [below = 1mm of 23]{23};
\node at (34) [right = 1mm of 34]{34};
\node at (14) [above = 1mm of 14]{14};
\node at (123) [below = 1mm of 123]{123};
\node at (234) [below = 1mm of 234]{234};
\node at (134) [above right = 1mm of 134]{134};
\node at (124) [above left = 1mm of 124]{124};
\node at (1234) [above = 1mm of 1234]{1234};

\node at (13) [above = 1mm of 13]{13};

\end{scope}

\end{tikzpicture}
\]
\end{figure}
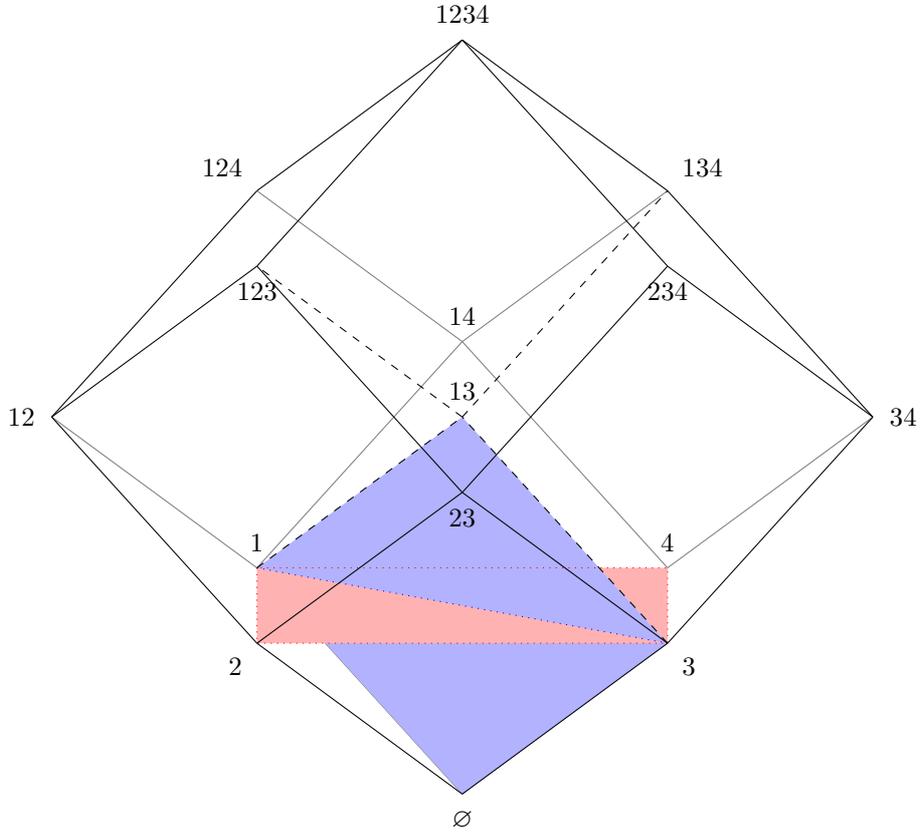

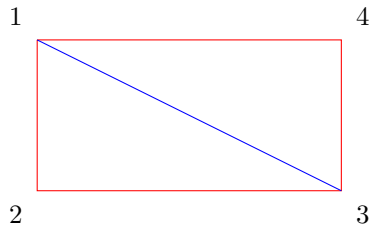
\begin{figure}
\caption{The triangulation $g(\mathcal{Q}_{2})=\mathcal{T}_{2}$ of $C(4, 2)$.}\label{fig:t2}
\[
\begin{tikzpicture}

% Coordinates for vertices

\coordinate(1) at (-2,1);
\coordinate(2) at (-2,-1);
\coordinate(3) at (2,-1);
\coordinate(4) at (2,1);

% Drawing edges and diagonal

\draw[red] (1) -- (2) -- (3) -- (4) -- (1);
\draw[blue] (1) -- (3);

% Labelling vertices

\node at (1) [above left = 1mm of 1]{1};
\node at (2) [below left = 1mm of 2]{2};
\node at (3) [below right = 1mm of 3]{3};
\node at (4) [above right = 1mm of 4]{4};

\end{tikzpicture}
\]
\end{figure}
\end{example}

\begin{remark}\label{rmk:g_dual}
There is a dual version of the map $g$, given by
\begin{align*}
\overline{g}\colon \mathcal{B}(n,\delta+1)&\rightarrow\mathcal{S}(n,\delta) \\
\mathcal{Q} &\mapsto \mathcal{Q}\cap \overline{H}_{n},
\end{align*}
where $\overline{H}_{n} = \{(x_{1}, \dots, x_{n})\in \mathbb{R}^{n} \mid x_{1} = n - 1 \}$. Given a cubillage $\mathcal{Q}$ of $Z(n,\delta+1)$, the triangulation $\overline{g}(\mathcal{Q})$ is induced by taking the vertex figure of $Z(n, n)$ at the vertex $\xi_{[n]}$. This map was considered in \cite[Proposition 7.1]{thomas-bst}. The dual of Proposition~\ref{prop-comb-interp} gives that if $\mathcal{Q} \lessdot \mathcal{Q}'$, then either $\overline{g}(\mathcal{Q})=\overline{g}(\mathcal{Q}')$ or $\overline{g}(\mathcal{Q}) \gtrdot \overline{g}(\mathcal{Q}')$. Hence $\overline{g}$ is order-reversing. That is, if $\mathcal{Q} \leqslant \mathcal{Q}'$, then $g(\mathcal{Q}) \geqslant g(\mathcal{Q}')$.
\end{remark}

\subsection{Separated collections}\label{sect:g_comb}

Our second definition of the map uses the characterisation of cubillages in terms of separated collections and the combinatorial framework for triangulations of cyclic polytopes from \cite{ot,njw-hst}. This is the framework we use to prove that $g$ is a quotient map of posets in Section~\ref{sect-surj} and Section~\ref{sect-quot}.

Given a triangulation $\mathcal{T}$ of $C(n,\delta)$, let \[\Sigma(\mathcal{T}):=\{A \subseteq [n] \mid |A| \text{ is a simplex of }\mathcal{T}\}.\] This can be viewed as the abstract simplicial complex corresponding to $\mathcal{T}$. The following lemma tells us how the value of $g(\mathcal{Q})$ is determined by $\mathrm{Sp}(\mathcal{Q})$.

\begin{lemma}\label{lem-all-simps}
Let $\mathcal{Q}$ be a cubillage of $Z(n,\delta+1)$ and $\mathcal{T}$ be a triangulation of $C(n,\delta)$. Then $g(\mathcal{Q})=\mathcal{T}$ if and only if $\mathrm{Sp}(\mathcal{Q})\supseteq \Sigma(\mathcal{T})$.
\end{lemma}
\begin{proof}
Suppose that $g(\mathcal{Q})=\mathcal{T}$. Let $|A|$ be a $\delta$-simplex of $\mathcal{T}$. Then there is a cube $\Delta$ of $\mathcal{Q}$ such that $|A|=\Delta \cap H_{n}$. We must have that the initial vertex of $\Delta$ is $\xi_{\emptyset}$ and the set of generating vectors is $A$. Thus if $|B|$ is a face of $|A|$, then $\xi_{B}$ is a vertex of $\Delta$, and hence $B \in \mathrm{Sp}(\mathcal{Q})$. Since every simplex of the triangulation $\mathcal{T}$ is a face of a $\delta$-simplex, we have that $\mathrm{Sp}(\mathcal{Q}) \supseteq \Sigma(\mathcal{T})$.

Conversely, suppose that $\mathrm{Sp}(\mathcal{Q}) \supseteq \Sigma(\mathcal{T})$. Let $|A|$ be a $\delta$-simplex of $\mathcal{T}$. Then $2^{A} \subseteq \Sigma(\mathcal{T}) \subseteq \mathrm{Sp}(\mathcal{Q})$. By \cite[(2.5)]{dkk}, the cube $\Delta$ with initial vertex $\emptyset$ and generating vectors $A$ is therefore a cube of $\mathcal{Q}$. This means that $|A|$ is a $\delta$-simplex of $g(\mathcal{Q})$, since $|A| = \Delta \cap H_{n}$. Since this is true for any $\delta$-simplex of $\mathcal{T}$, we must have $g(\mathcal{Q})=\mathcal{T}$.
\end{proof}

In fact, as the following proposition shows, we need only consider $\intsp(\mathcal{Q}) \cap \binom{[n]}{\floor{\delta/2} + 1}$ to know the value of $g(\mathcal{Q})$.

\begin{proposition}\label{prop-comb-interp}
Given a cubillage $\mathcal{Q} \in \mathcal{B}(n,\delta+1)$, we have that $\simp(g(\mathcal{Q}))=\intsp(\mathcal{Q}) \cap \binom{[n]}{\floor{\delta/2} + 1}$.
\end{proposition}
\begin{proof}
It follows immediately from Lemma~\ref{lem-all-simps} that $\simp(g(\mathcal{Q})) \subseteq \intsp(\mathcal{Q}) \cap \binom{[n]}{\floor{\delta/2} + 1}$, since if $\# A = \floor{\delta/2} + 1$, then $|A|$ is an internal $\floor{\delta/2}$-simplex in $C(n,\delta)$ if and only if $\xi_{A}$ is an internal point in $Z(n,\delta+1)$, by Observation~\ref{obs:internal}.

To show that $\simp(g(\mathcal{Q})) \supseteq \intsp(\mathcal{Q}) \cap \binom{[n]}{\floor{\delta/2} + 1}$, suppose that we have $A \in \left(\intsp(\mathcal{Q}) \cap \binom{[n]}{\floor{\delta/2} + 1}\right) \setminus \simp(g(\mathcal{Q}))$. Then note that $|A|$ must be an internal $\floor{\delta/2}$-simplex in $C(n, \delta)$, since $\xi_{A}$ is an internal point in $Z(n, \delta + 1)$. However, $|A|$ is not a $\floor{\delta/2}$-simplex of $\mathcal{T}$, so $\pi_{n - 1, \delta}|A|$ intersects a $\ceil{\delta/2}$-simplex $\pi_{n - 1, \delta}|B|$ of $\pi_{n - 1, \delta}(\mathcal{T})$ transversely. This implies that $(A, B)$ is a circuit, and so $A$ and $B$ are $\delta$-interweaving. But this is a contradiction, since $B \in \mathrm{Sp}(\mathcal{Q})$ by Lemma~\ref{lem-all-simps}.
\end{proof}

Proposition~\ref{prop-comb-interp} gives an interpretation of the map $g$ in terms of separated collections. We know that a cubillage $\mathcal{Q}$ of $Z(n,\delta+1)$ is determined by $\intsp(\mathcal{Q})$, and likewise a triangulation $\mathcal{T}$ of $C(n,\delta)$ is determined by $\simp(\mathcal{T})$. Hence one could also define $g(\mathcal{Q})$ to be the triangulation $\mathcal{T}$ such that $\simp(\mathcal{T})=\intsp(\mathcal{Q}) \cap \binom{[n]}{\floor{\delta/2} + 1}$.

\begin{example}\label{ex:g_comb}
We illustrate how to apply the interpretation of $g$ from Proposition~\ref{prop-comb-interp} to the cubillages from Example~\ref{ex:vertex_figures}.

Consider the internal spectrum of $\mathcal{Q}_{1}$, as shown in Figure~\ref{fig:q1}. We have $\intsp(\mathcal{Q}_{1}) = \{3,13,23\}$, so $\intsp(\mathcal{Q}_{1}) \cap \binom{[4]}{1} = \{3\}$. This implies that $\{3\}=\simp(g(\mathcal{Q}_{1}))=\simp(\mathcal{T}_{1})$, which is indeed the case. Note that having $\simp(\mathcal{T}_{1})=\{3\}$ defines $\mathcal{T}_{1}$.

Next, consider the internal spectrum of $\mathcal{Q}_{2}$, as shown in Figure~\ref{fig:q2}. We have $\intsp(\mathcal{Q}_{2}) = \{13\}$, so $\intsp(\mathcal{Q}_{2}) \cap \binom{[4]}{2} = \{13\}$. This implies that $\{13\}=\simp(g(\mathcal{Q}_{2}))=\simp(\mathcal{T}_{2})$, which is indeed the case. Note that having $\simp(\mathcal{T}_{2})=\{13\}$ defines $\mathcal{T}_{2}$.
\end{example}

\begin{remark}
The interpretation of $\overline{g}$ for separated collections is as follows. We have that $\overline{g}(\mathcal{Q})$ is the triangulation $\mathcal{T}$ such that \[\simp(\mathcal{T}) = \left\lbrace [n]\setminus A \mathrel{\Big|}  A \in \mathcal{C}\cap \subs{n}{n-\floor{\delta/2} - 1}{}\right\rbrace.\]
\end{remark}

\subsection{Admissible orders}\label{sect:vis}

In this section we give a way of defining the map $g$ while interpreting the elements of the higher Bruhat orders as equivalence classes of admissible orders. We use the following notions, which were used in \cite{dm-h} to define the higher Tamari orders.

Let $\alpha$ be an admissible order of $\subs{n}{\delta + 1}{}$ and $I \in \subs{n}{\delta + 1}{}$. Given $e \in [n]\setminus I$, we say that $I$ is \emph{invisible in $P(I \cup \{e\})$} if either
\begin{itemize}
\item $I \cup \{e\} \notin \mathrm{inv}(\alpha)$ and $e$ is an odd gap in $I$, or
\item $I \cup \{e\} \in \mathrm{inv}(\alpha)$ and $e$ is an even gap in $I$.
\end{itemize}
Otherwise, we say that $I$ is \emph{coinvisible in $P(I \cup \{e\})$}. (We note that $I$ being invisible in $P(I \cup \{e\})$ is equivalent to $e$ being \emph{externally semi-active} with respect to $I$, in the terminology of \cite{gpw}, which applies to more general matroids.)

Then:
\begin{itemize}
\item We say that $I$ is \emph{invisible in $\alpha$} if there is a $e \in [n]\setminus I$ such that $I$ is invisible in $P(I \cup \{e\})$.
\item We say that $I$ is \emph{coinvisible in $\alpha$} if there is a $e \in [n]\setminus I$ such that $I$ is coinvisible in $P(I \cup \{e\})$.
\item We say that $I$ is \emph{visible in $\alpha$} if there is no $e \in [n]\setminus I$ such that $I$ is invisible in $P(I \cup \{e\})$. (Note that this is not the same notion of visibility as in \cite[Section 9]{dkk-survey}.)
\item We say that $I$ is \emph{covisible in $\alpha$} if there is no $e \in [n]\setminus I$ such that $I$ is coinvisible in $P(I \cup \{e\})$.
\end{itemize}

Given an admissible order $\alpha$ of $\binom{[n]}{\delta + 1}$, we use $V(\alpha)$ to denote the elements of $\binom{[n]}{\delta+1}$ which are visible in $\alpha$ and $\overline{V}(\alpha)$ to denote the elements of $\binom{[n]}{\delta+1}$ which are covisible in $\alpha$. (In \cite{dm-h-simplex}, visible elements are labelled in blue; covisible elements are labelled in red; and elements which are neither visible nor covisible are labelled in green.)

Given an admissible order $\alpha$ of $\binom{[n]}{\delta+1}$, we write $\mathcal{Q}_{\alpha}$ for the corresponding cubillage of $Z(n,\delta+1)$.

\begin{proposition}\label{prop-vis-init-vert}
Let $\alpha$ be an admissible order of $\subs{n}{\delta+1}{}$ and $I \in \subs{n}{\delta+1}{}$. Then the cube in $\mathcal{Q}_{\alpha}$ with generating set $I$ has initial vertex $\xi_{E}$, where \[E = \{e \in [n]\setminus I \mid I \text{ is invisible in } P(I \cup \{e\} )\}.\]
\end{proposition}
\begin{proof}
This follows immediately from the correspondence between admissible orders and cubillages in \cite{thomas-bst}, as described in Section~\ref{sect-hbo}.
\end{proof}

The following result was noted in \cite[Appendix B]{dkk-survey}.

\begin{corollary}\label{cor-vis-init-vert}
Let $\alpha$ be an admissible order of $\subs{n}{\delta+1}{}$ and $I \in \subs{n}{\delta+1}{}$. Then $I \in V(\alpha)$ if and only if the cube in $\mathcal{Q}_{\alpha}$ with generating set $I$ has initial vertex $\xi_{\emptyset}$.
\end{corollary}

This gives us yet another interpretation of the map $g$.

\begin{corollary}\label{cor:g_vis_interp}
Given $[\alpha] \in \mathcal{B}(n,\delta+1)$, we have that $g(\mathcal{Q}_{\alpha})$ is the triangulation with \[\{|A| \mid A \in V(\alpha)\}\] as its set of $\delta$-simplices.
\end{corollary}

\begin{example}
We continue from Example~\ref{ex:vertex_figures} and Example~\ref{ex:g_comb} and illustrate how the map $g$ can also be characterised using visibility.

We consider $\mathcal{Q}_{1}$ first. By labelling the cubes of $\mathcal{Q}_{1}$ with the elements of $\binom{[4]}{2}$, as shown in Figure~\ref{fig:q1_cubes}, it can be seen that the admissible order corresponding to $\mathcal{Q}_{1}$ is \[\alpha_{1} = \{23<13<12<14<24<34\}.\] We compute that $\mathrm{inv}(\alpha_{1})=\{123\}$.

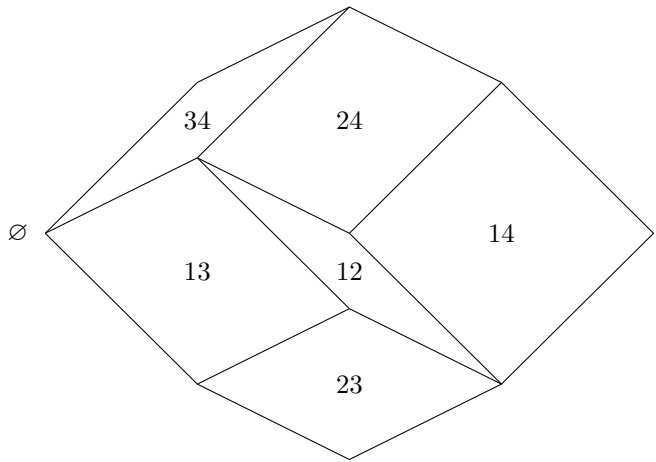
\begin{figure}
\caption{$\mathcal{Q}_{1}$ with its cubes labelled.}\label{fig:q1_cubes}
\[
\begin{tikzpicture}

% Boundary vertices

\coordinate(0) at (0,0);
\node at (0)[left = 1mm of 0]{$\emptyset$};
\coordinate(1) at (2,-2);
%\node at (1)[below left = 1mm of 1]{1};
\coordinate(12) at (4,-3);
%\node at (12)[below = 1mm of 12]{12};
\coordinate(123) at (6,-2);
%\node at (123)[below right = 1mm of 123]{123};
\coordinate(1234) at (8,0);
%\node at (1234)[right = 1mm of 1234]{1234};
\coordinate(234) at (6,2);
%\node at (234)[above right = 1mm of 234]{234};
\coordinate(34) at (4,3);
%\node at (34)[above = 1mm of 34]{34};
\coordinate(4) at (2,2);
%\node at (4)[above left = 1mm of 4]{4};

% Boundary edges

\draw (0) -- (1) -- (12) -- (123) -- (1234) -- (234) -- (34) -- (4) -- (0);

% Internal vertices

\coordinate(3) at (2,1);
\coordinate(13) at (4,-1);
\coordinate(23) at (4,0);

% Labels of cubes

\node at (2,-0.5){13};
\node at (2,1.5){34};
\node at (4,-2){23};
\node at (4,-0.5){12};
\node at (4,1.5){24};
\node at (6,0){14};

% Internal edges

\draw (0) -- (3);
\draw (3) -- (34);
\draw (3) -- (23);
\draw (3) -- (13);
\draw (1) -- (13);
\draw (13) -- (123);
\draw (23) -- (123);
\draw (23) -- (234);

\end{tikzpicture}
\]
\end{figure}

We can then analyse which elements of $\binom{[4]}{2}$ are visible in $\alpha_{1}$:
\begin{itemize}
\item 23: invisible because $123 \in \mathrm{inv}(\alpha_{1})$ and $1$ is an even gap in $23$;
\item 13: visible;
\item 12: invisible because $123 \in \mathrm{inv}(\alpha_{1})$ and $3$ is an even gap in $12$;
\item 14: invisible because $124 \notin \mathrm{inv}(\alpha_{1})$ and $2$ is an odd gap in $14$;
\item 24: invisible because $234 \notin \mathrm{inv}(\alpha_{1})$ and $3$ is an odd gap in $24$;
\item 34: visible. 
\end{itemize}
Note that, as Corollary~\ref{cor-vis-init-vert} shows, 13 and 34 are precisely the cubes with $\xi_{\emptyset}$ as their initial vertex. Furthermore, as Corollary~\ref{cor:g_vis_interp} shows, $g(\mathcal{Q}_{1})=\mathcal{T}_{1}$ is the triangulation with $1$-simplices $|13|$ and $|34|$.

We now conduct the same analysis of $\mathcal{Q}_{2}$. The admissible order corresponding to $\mathcal{Q}_{2}$ is \[\alpha_{2}=\{123<124<134<234\}.\] It is easy to see that $\mathrm{inv}(\alpha_{2})=\emptyset$. Hence the visible elements of $\binom{[4]}{3}$ in $\alpha_{2}$ are as follows:
\begin{itemize}
\item 123: visible;
\item 124: invisible because $1234 \notin \mathrm{inv}(\alpha_{2})$ and $3$ is an odd gap in $124$;
\item 134: visible;
\item 234: invisible because $1234 \notin \mathrm{inv}(\alpha_{2})$ and $1$ is an odd gap in $234$.
\end{itemize}
Again, it can be seen in Figure~\ref{fig:q2} that 123 and 134 are precisely the cubes with $\xi_{\emptyset}$ as their initial vertex, as shown by Corollary~\ref{cor-vis-init-vert}. Moreover, as Corollary~\ref{cor:g_vis_interp} shows, $g(\mathcal{Q}_{2})=\mathcal{T}_{2}$ is the triangulation with $2$-simplices $|123|$ and $|134|$.
\end{example}

The dual statements to Proposition~\ref{prop-vis-init-vert}, Corollary~\ref{cor-vis-init-vert}, and Corollay~\ref{cor:g_vis_interp} are as follows.

\begin{proposition}
Let $\alpha$ be an admissible order of $\subs{n}{\delta+1}{}$ and $I \in \subs{n}{\delta+1}{}$. Then the cube in $\mathcal{Q}_{\alpha}$ with generating set $I$ has final vertex $\xi_{F}$ where \[F = [n]\setminus\{e \in [n]\setminus I \mid I \text{ is coinvisible in } P(I \cup \{e\})\}.\]
\end{proposition}

\begin{corollary}
Let $\alpha$ be an admissible order of $\subs{n}{\delta+1}{}$ and $I \in \subs{n}{\delta+1}{}$. Then $I \in \overline{V}(\alpha)$ if and only if the cube in $\mathcal{Q}_{\alpha}$ with generating set $I$ has final vertex $\xi_{[n]}$.
\end{corollary}

\begin{corollary}
Given $[\alpha] \in \mathcal{B}(n,\delta+1)$, we have that $\overline{g}(\mathcal{Q}_{\alpha})$ is the triangulation with \[\{|A| \mid A \in \overline{V}(\alpha)\}\] as its set of $\delta$-simplices.
\end{corollary}

\section{Quotient maps of posets}\label{sect:quotient_framework}

Dimakis and M\"uller-Hoissen use the definition of the map $g$ from Section~\ref{sect:vis} to define the \emph{higher Tamari orders}. We restate their definition in the framework of quotient posets. In this section, we explain our approach to this notion.

Given a poset $(X, \leqslant)$ subject to an equivalence relation $\sim$, the \emph{quotient} $(X/{\sim}, R)$ is defined to be the set of $\sim$-equivalence classes $[x]$ of $X$, with the binary relation $R$ defined by $[x]R[y]$ if and only if there exist $x' \in [x]$ and $y' \in [y]$ such that $x' \leqslant y'$. The quotient of a poset is in general only a reflexive binary relation, not a partial order, since the relation $R$ is not necessarily transitive or anti-symmetric.

Previous authors have considered various different conditions on the equivalence relation $\sim$ which are sufficient to guarantee that the quotient $X/{\sim}$ is a poset. Stanley considers the case where $\sim$ is given by the orbits of a group of automorphisms \cite{stanley_quotient_peck,stanley-appl-alg-comb}. Two similar notions of congruence which also preserve lattice-theoretic properties are considered by Chajda and Sn\'a\v{s}el, and Reading \cite{cs_cong,reading_order}. Most recently, Hallam and Sagan \cite{hs-char-poly,hallam_applications} consider \emph{homogeneous quotients} in order to study the characteristic polynomials of lattices.

Whilst these conditions are sufficient to guarantee that the quotient poset is well-defined, none of them are necessary. In this paper we are interested only in the minimal conditions which provide that the quotient poset is well-defined, and not in whether the quotient also preserves other properties. These necessary and sufficient conditions are as follows.

\begin{proposition}\label{prop:tautology}
The quotient $X/{\sim}$ is a poset if and only if
\begin{enumerate}
\item if there exist $x_{1} \sim x$ and $y_{1} \sim y$ such that $x_{1} \leqslant y_{1}$, and $x_{2} \sim x$ and $y_{2} \sim y$ such that $x_{2} \geqslant y_{2}$, then $x \sim y$, and\label{op:antisym}
\item given $x, y, z \in X$ such that there exist $x_{1} \sim x$ and $y_{1} \sim y$ such that $x_{1} \leqslant y_{1}$, and $y_{2} \sim y$ and $z_{2} \sim z$ such that $y_{2} \leqslant z_{2}$, then there exist $x_{3} \sim x$ and $z_{3} \sim z$ such that $x_{3} \leqslant z_{3}$.\label{op:trans}
\end{enumerate}
\end{proposition}
\begin{proof}
Condition (\ref{op:antisym}) is equivalent to the relation $R$ being anti-symmetric. Condition (\ref{op:trans}) is equivalent to the relation $R$ being transitive.
\end{proof}

If both condition (\ref{op:antisym}) and condition (\ref{op:trans}) hold, then we write $\leqslant$ instead of $R$, to acknowledge that the relation gives us a partial order. In this case, we say that $\sim$ is a \emph{weak order congruence} on the poset $X$. Note that, in particular, order congruences \cite{reading_order,cs_cong} and the equivalence relations which give homogeneous quotients \cite{hs-char-poly,hallam_applications} are  weak order congruences.

If $\sim$ is a weak order congruence, so that $X/{\sim}$ is a poset, then we have a canonical order-preserving map
\begin{align*}
X &\rightarrow X/{\sim} \\
x &\mapsto [x].
\end{align*}
Indeed, for any order-preserving map of posets $f \colon X \to Y$, one can consider the equivalence relation on $X$ defined by $x \sim x'$ if and only if $f(x) = f(x')$. We then define the \emph{image} of $f$ to be the quotient $f(X) = X/{\sim}$. We identify the $\sim$-equivalence class $[x]$ of $X$ with the element $f(x) \in Y$, so that $f(X) \subseteq Y$ and the quotient relation on $f(X)$ is a subrelation of the partial order on $Y$. If the equivalence relation $\sim$ on $X$ is a weak order congruence, so that the image $f(X)$ is a well-defined poset, then we say that the map $f$ is \emph{photogenic}.

We say that a map $f \colon X \to Y$ is \emph{full} if whenever $f(x_{1}) \leqslant f(x_{2})$ in $Y$, there exist $x_{1}', x'_{2} \in X$ such that $x'_{1} \leqslant x'_{2}$, with $f(x'_{1}) = f(x_{1})$ and $f(x'_{2}) = f(x_{2})$. (In \cite{cs_cong}, maps which are full and order-preserving are called \emph{strong}.)

\begin{proposition}\label{prop:quotient_maps}
Let $X$ and $Y$ be posets with $f \colon X \to Y$ an order-preserving map. Then the relation on $f(X)$ is anti-symmetric. Furthermore, if $f$ is full, then the relation on $f(X)$ is transitive, and so $f$ is photogenic. Finally, $f(X) = Y$ as posets if and only if $f$ is surjective and full.
\end{proposition}
\begin{proof}
Suppose that $x_{1}, x_{2} \in X$ are such that $[x_{1}]R[x_{2}]$ and $[x_{2}]R[x_{1}]$. Since $f$ is order-preserving, this implies that $f(x_{1}) \leqslant f(x_{2})$ and $f(x_{2}) \leqslant f(x_{1})$. Hence $f(x_{1}) = f(x_{2})$ and so $x_{1} \sim x_{2}$. Thus $R$ is anti-symmetric. 

Now suppose that $f$ is full. Let $x_{1}, x_{2}, x_{3} \in X$ be such that $[x_{1}]R[x_{2}]$ and $[x_{2}]R[x_{3}]$. This implies that $f(x_{1}) \leqslant f(x_{2})$ and $f(x_{2}) \leqslant f(x_{3})$, since $f$ is order-preserving. Hence $f(x_{1}) \leqslant f(x_{3})$. Since $f$ is full, there exist $x'_{1}, x'_{3} \in X$ such that $x'_{1} \leqslant x'_{3}$, with $f(x'_{1}) = f(x_{1})$ and $f(x'_{3}) = f(x_{3})$. Hence $[x_{1}]R[x_{3}]$, and so $R$ is transitive.

Finally, it is clear that $f(X) = Y$ as sets if and only if $f$ is surjective. Then $f$ being full and order-preserving is equivalent to having $[x_{1}] \leqslant [x_{2}]$ in $f(X)$ if and only if $f(x_{1}) \leqslant f(x_{2})$ in $Y$.
\end{proof}

Therefore, every quotient of a poset by a weak order congruence gives an order-preserving map which is surjective and full, and, conversely, every order-preserving map which is surjective and full gives a quotient by a weak order congruence. Hence, if an order-preserving map $f$ is surjective and full, then we say that $f$ is a \emph{quotient map of posets}.

With this technical framework in mind, the \emph{higher Tamari order} $T(n,\delta+1)$ \cite{dm-h} is defined to be the image of the map $g \colon \mathcal{B}(n, \delta + 1) \to \mathcal{S}(n, \delta)$, or, explicitly, the quotient of $\mathcal{B}(n, \delta + 1)$ by the relation defined by $\mathcal{Q} \sim \mathcal{Q}'$ if and only if $g(\mathcal{Q}) = g(\mathcal{Q}')$. That this is equivalent to \cite[Definition 4.7]{dm-h} follows from Corollary~\ref{cor:g_vis_interp}. Note that it is not evident that $T(n, \delta + 1)$ is a well-defined poset, since it is not clear that the map $g$ is photogenic. However, in Section~\ref{sect-quot} we shall prove that $g$ is full, which implies that $g$ is photogenic by Proposition~\ref{prop:quotient_maps}, since we already know that $g$ is order-preserving by Corollary~\ref{cor-ord-pres}. In Section~\ref{sect-surj}, we give a new proof of the fact that $g$ is surjective, originally known from \cite[Theorem 3.5]{rs-baues}. Therefore, the results of the two subsequent sections entail the following theorem. 

\begin{theorem}\label{thm:quot}
The map $g \colon \mathcal{B}(n, \delta + 1) \to \mathcal{S}(n, \delta)$ is a quotient map of posets.
\end{theorem}

Hence, we obtain by Proposition~\ref{prop:quotient_maps} that the higher Tamari orders are indeed the same posets as the first higher Stasheff--Tamari orders.

\begin{corollary}\label{cor:t=st}
$T(n,\delta+1) \cong \mathcal{S}(n,\delta)$.
\end{corollary}

\section{Surjectivity}\label{sect-surj}

We now give a new construction showing that the map $g$ is a surjection. Our strategy is to explicitly show that $g$ is a surjection when $\delta$ is even, and then to use this to deduce the case where $\delta$ is odd. Given a triangulation $\mathcal{T}$ of $C(n,2d)$, we will construct a cubillage $\mathcal{Q}_{\mathcal{T}}$ of $Z(n,2d+1)$ such that $g(\mathcal{Q}_{\mathcal{T}})=\mathcal{T}$. We will define $\mathcal{Q}_{\mathcal{T}}$ by specifying its internal spectrum.

\begin{convention}
In this section and in Section~\ref{sect-quot}, we will frequently be using arithmetic modulo $n$. In particular, given a set $S \in \binom{[n]}{2d + 2}$, we have $s_{0} - s_{2d + 1} \equiv s_{0} - s_{2d + 1} + n \pmod{n}$, which is an element of $[n]$.
\end{convention}

For $I \subseteq [n]$, we write $I = J \sqcup J'$ if $I = J \cup J'$ and there are no $j \in J, j' \in J'$ such that $j, j'$ are cyclically consecutive. Given a cyclic $l$-ple interval $I = [i_{0}, i'_{0}] \sqcup \dots \sqcup [i_{l-1}, i'_{l-1}]$, we use the notation $\widehat{I} := \{i_{0}, \dots, i_{l-1}\}$ from \cite{njw-jm}. We claim that the collection of subsets \[U(\mathcal{T})=\left\lbrace I \subseteq [n]\mid |\widehat{I}| \text{ is a } d'\text{-simplex of }\mathcal{T} \text{ for }d' \geqslant d\right\rbrace \] defines the internal spectrum of a cubillage on $Z(n, 2d + 1)$. This is similar to the construction in \cite[Theorem 3.8]{njw-jm}. In order to show that $U(\mathcal{T})$ is the internal spectrum of a cubillage, we must show that it is $2d$-separated and that $\# U(\mathcal{T})=\binom{n-1}{2d+1}$. We begin by showing that $U(\mathcal{T})$ is $2d$-separated, for which we need the following lemma. This generalises one direction of \cite[Lemma 3.7]{njw-jm}, although the proof in \textit{op.\ cit.}\ requires only minor changes. 

\begin{lemma}\label{lem-int-endpts}
Let $I,J \subseteq [n]$. Then $I$ $\delta$-interweaves $J$ only if there exist subsets $X \subseteq \widehat{I}$ and $Y \subseteq \widehat{J}$ such that $\# X = \floor{\delta/2}$ and $\# Y = \ceil{\delta/2}$, and $X$ $\delta$-interweaves $Y$. 
\end{lemma}
\begin{proof}
We let $\delta=2d$, since the case $\delta=2d+1$ behaves similarly.

Let $I = [i_{0}, i'_{0}] \sqcup \dots \sqcup [i_{r}, i'_{r}]$ and $J = [j_{0}, j'_{0}] \sqcup \dots \sqcup [j_{s}, j'_{s}]$. Suppose that $I$ $2d$-interweaves $J$, and let $A \subseteq I\setminus J$ and $B \subseteq J\setminus I$ witness this. For any $0\leqslant p < q\leqslant d$ we cannot have both $a_{p} \in [i_{t}, i'_{t}]$ and $a_{q} \in [i_{t}, i'_{t}]$, since this implies that $b_{p}, \dots, b_{q-1} \in [i_{t}, i'_{t}] \subseteq I$, which contradicts $B \cap I = \emptyset$. Hence, for all $0 \leqslant k \leqslant d$, let $t_{k}$ be such that $a_{k} \in [i_{t_{k}}, i'_{t_{k}}]$ and let $u_{k}$ be such that $b_{k} \in [j_{u_{k}}, j'_{u_{k}}]$. Moreover, since $B \cap I = \emptyset$, we have $b_{k} \in (i'_{t_{k}}, i_{t_{k+1}})$, and similarly $a_{k} \in (j'_{u_{k-1}}, j_{u_{k}})$ for $k \in \mathbb{Z}/(d + 1)\mathbb{Z}$. Then \[i_{t_{0}} \leqslant a_{0} < j_{u_{0}} \leqslant b_{0} < i_{t_{1}} \leqslant a_{1} < \dots < i_{t_{d}} \leqslant a_{d} < j_{u_{d}} \leqslant b_{d},\] and so \[ i_{t_{0}} < j_{u_{0}} < i_{t_{1}} < \dots < i_{t_{d}} < j_{u_{d}}.\] Letting $X=\{i_{t_{0}}, \dots, i_{t_{d}}\}$ and $Y=\{j_{u_{0}}, \dots, j_{u_{d}}\}$ gives us the desired result.
\end{proof}

\begin{lemma}\label{lem:cub_sep}
The collection $U(\mathcal{T})$ is $2d$-separated.
\end{lemma}
\begin{proof}
Suppose that there exist $I,J \in U(\mathcal{T})$ such that $I$ and $J$ are $2d$-interweaving. By Lemma~\ref{lem-int-endpts}, we have $X \subseteq \widehat{I}$ and $Y \subseteq \widehat{J}$ such that $X$ and $Y$ are $\delta$-interweaving. But this implies that $\widehat{I}$ and $\widehat{J}$ each contain one half of a circuit $(X, Y)$ for $C(n, 2d)$. This is a contradiction, since, by construction of $U(\mathcal{T})$, $|\widehat{I}|$ and $|\widehat{J}|$ are both simplices of the triangulation $\mathcal{T}$ of $C(n, 2d)$.
\end{proof}

We must now show that $\# U(\mathcal{T})=\binom{n-1}{2d+1}$. We use induction for this, showing that the size of $U(\mathcal{T})$ is preserved by increasing flips of $\mathcal{T}$, which requires the following lemma.

\begin{lemma}\label{lem-card-flip}
Let $|S|$ be a $(2d + 1)$-simplex inducing an increasing flip of a triangulation $\mathcal{T}$ of $C(n, 2d)$ and denote $S_{l} = \{s_{0}, s_{2}, \dots, s_{2d}\}$ and $S_{u} = \{s_{1}, s_{3}, \dots, s_{2d + 1}\}$. Then the following two sets have the same cardinality:
\begin{align*}
\mathcal{I}_{l}(S, n) &= \left\lbrace I \subseteq [n] \mid S_{l} \subseteq \widehat{I} \subset S \right\rbrace, \\ 
\mathcal{I}_{u}(S, n) &= \left\lbrace I \subseteq [n] \mid S_{u} \subseteq \widehat{I} \subset S \right\rbrace.
\end{align*}
\end{lemma}
Here we use the symbol `$\subset$' to denote proper subsets.
\begin{proof}
Note that we may instead consider
\begin{align*}
\mathcal{I}'_{l}(S,n)&:=\left\lbrace I \subseteq [n] \mid S_{l} \subseteq \widehat{I} \subseteq S\right\rbrace , \\
\mathcal{I}'_{u}(S,n)&:=\left\lbrace I \subseteq [n] \mid S_{u} \subseteq \widehat{I} \subseteq S\right\rbrace .
\end{align*}
This is because \[\mathcal{I}'_{l}(S,n)\setminus \mathcal{I}_{l}(S,n)=\mathcal{I}'_{u}(S,n)\setminus \mathcal{I}_{u}(S,n)=\left\lbrace I \subseteq [n] \mid \widehat{I} = S\right\rbrace .\] Hence if $\# \mathcal{I}'_{l}(S, n) = \# \mathcal{I}'_{u}(S, n)$, then $\# \mathcal{I}_{l}(S, n) = \# \mathcal{I}_{u}(S, n)$.

We prove the claim by explicit enumeration. Let \[I=[s_{0}, s'_{0}] \cup [s_{1}, s'_{1}] \cup [s_{2}, s'_{2}] \cup \dots \cup [s_{2d}, s'_{2d}] \cup [s_{2d + 1}, s'_{2d + 1}].\] Then $I \in \mathcal{I}'_{l}(S, n)$ if and only if, for all $i \in \mathbb{Z}/(d + 1)\mathbb{Z}$, \[s'_{2i} \in [s_{2i}, s_{2i + 1} - 1] \text{ and } s'_{2i + 1} \in [s_{2i + 1} - 1, s_{2i + 2} - 2].\] Recall that our convention here is that if $s'_{j} = s_{j} - 1$, then $[s_{j}, s'_{j}] = \emptyset$. Similarly, $I \in \mathcal{I}'_{u}(S, n)$ if and only if, for all $i \in \mathbb{Z}/(d + 1)\mathbb{Z}$, \[s'_{2i} \in [s_{2i} - 1, s_{2i + 1} - 2] \text{ and } s'_{2i + 1} \in [s_{2i + 1}, s_{2i + 2} - 1].\] Therefore, 
\begin{align*}
\# \mathcal{I}'_{l}(S, n) = \# \mathcal{I}'_{u}(S, n) &= \prod_{i \in \mathbb{Z}/(d + 1)\mathbb{Z}}(s_{2i + 1} - s_{2i})(s_{2i + 2} - s_{2i + 2})\\
&= \prod_{j \in \mathbb{Z}/(2d + 2)\mathbb{Z}}(s_{j + 1} - s_{j}).
\end{align*}
\end{proof}

This allows us to prove that our $2d$-separated collection $U(\mathcal{T})$ is the right size to be the internal spectrum of a cubillage.

\begin{lemma}\label{lem:right_size}
Given a triangulation $\mathcal{T}$ of $C(n,2d)$, we have that $\# U(\mathcal{T})=\binom{n-1}{2d+1}$. 
\end{lemma}
\begin{proof}
We prove the claim by induction on increasing flips of the triangulation. This is valid since every triangulation of a cyclic polytope can be reached via a sequence of increasing flips from the lower triangulation by \cite[Theorem 1.1(i)]{rambau}.

For the base case, let $\mathcal{T}_{l}$ be the lower triangulation of $C(n,2d)$. By Gale's Evenness Criterion, the $2d$-simplices of $\mathcal{T}_{l}$ are given by $1$ together with $d$ disjoint pairs of consecutive numbers. Therefore, the only $d'$-simplices of $\mathcal{T}_{l}$ with $d' \geqslant d$ which have no cyclically consecutive entries are the internal $d$-simplices. Hence if $I \in U(\mathcal{T}_{l})$, then $|\widehat{I}|$ is an internal $d$-simplex of $\mathcal{T}_{l}$. Moreover, the internal $d$-simplices of $\mathcal{T}_{l}$ are given by $(d+1)$-subsets which are cyclic $(d + 1)$-ple intervals and contain 1.

By \cite[(4.2)(ii)]{dkk}, the internal spectrum of the lower cubillage of $Z(n,2d+1)$ consists of all cyclic $(d + 1)$-ple intervals which contain 1. It is then straightforward to see that $\mathcal{U}(\mathcal{T})$ is indeed the internal spectrum of the lower cubillage of $Z(n, 2d + 1)$ when $\mathcal{T}$ is the lower triangulation of $C(n, 2d)$. Therefore, we have in this case that $\# U(\mathcal{T})=\binom{n-1}{2d+1}$.

For the inductive step, we suppose that we have a triangulation $\mathcal{T}'$ obtained by performing an increasing flip induced by a $(2d+1)$-simplex $|S|$ on a triangulation $\mathcal{T}$ for which the induction hypothesis holds. Then $\mathcal{I}_{l}(S, n)$ contains precisely the subsets $I$ such that $\pi_{n - 1, 2d + 1}|\widehat{I}|$ is contained in a lower facet of $\pi_{n - 1, 2d + 1}|S|$ but not any upper facets, by Gale's Evenness Criterion. Similarly, $\mathcal{I}_{u}(S, n)$ contains precisely the subsets $I$ such that $\pi_{n - 1, 2d + 1}|\widehat{I}|$ is contained in an upper facet of $\pi_{n - 1, 2d + 1}|S|$ but not any lower facets. Hence \[U(\mathcal{T}') = (U(\mathcal{T})\setminus\mathcal{I}_{l}(S, n)) \cup \mathcal{I}_{u}(S, n),\] and so $\# U(\mathcal{T}) = \# U(\mathcal{T}')$ by Lemma~\ref{lem-card-flip}. The result then follows by induction.
\end{proof}

Hence we obtain that $g$ is a surjection in even dimensions.

\begin{theorem}\label{thm-g-even}
The map $g \colon \mathcal{B}(n, \delta + 1) \to \mathcal{S}(n, \delta)$ is a surjection for even $\delta$. 
\end{theorem}
\begin{proof}
Let $\delta=2d$ and let $\mathcal{T}$ be a triangulation of $C(n,2d)$. By Lemma~\ref{lem:cub_sep}, Lemma~\ref{lem:right_size}, and the correspondence between cubillages and separated collections from \cite{gp}, we have that the collection $U(\mathcal{T})$ is the internal spectrum of a cubillage $\mathcal{Q}_{\mathcal{T}}$ of $Z(n,2d+1)$. Moreover, $g(\mathcal{Q}_{\mathcal{T}})=\mathcal{T}$ by Proposition~\ref{prop-comb-interp}, since if $\# A=d+1$, then $A \in U(\mathcal{T})$ if and only if $|A|$ is an internal $d$-simplex of $\mathcal{T}$.
\end{proof}

\begin{example}\label{ex:even_surj}
We give an example of the construction used to prove Theorem~\ref{thm-g-even}. Consider the triangulation $\mathcal{T}$ of the hexagon $C(6,2)$ which has arcs $\simp(\mathcal{T})=\{13,15,35\}$.

Then we have
\begin{align*}
U(\mathcal{T})=\{ &13,15,35,\\
&134,125,356,135,\\
&1345,1235,1356\}.
\end{align*}
Note the presence of $135 \in U(\mathcal{T})$, since $|135|$ is a 2-simplex of $\mathcal{T}$. One can check that $U(\mathcal{T})$ is 2-separated. Furthermore, $\# U(\mathcal{T}) = 10 = \binom{5}{3} = \binom{6-1}{2+1}$, as desired.

We thus obtain the cubillage $\mathcal{Q}_{\mathcal{T}}$ which is defined by $\intsp(\mathcal{Q}_{\mathcal{T}})=U(\mathcal{T})$. It then follows from Proposition~\ref{prop-comb-interp} that $g(\mathcal{Q}_{\mathcal{T}})=\mathcal{T}$; compare Example~\ref{ex:g_comb}. Hence $\mathcal{T}$ has a pre-image under $g$.
\end{example}

We now use this result to show that the map $g$ must be a surjection for odd $\delta$. Following many authors, given a set $\mathcal{S}$ of subsets of $[n]$, we denote by $\mathcal{S} \ast (n+1)$ the set \[\mathcal{S} \ast (n+1) = \{ A \cup \{n+1\} \mid A \in \mathcal{S}\}.\] 

\begin{theorem}\label{thm-g-odd}
The map $g \colon \mathcal{B}(n, \delta + 1) \to \mathcal{S}(n, \delta)$ is a surjection for odd $\delta$.
\end{theorem}
\begin{proof}
Let $\delta=2d+1$. Let $\mathcal{T}$ be a triangulation of $C(n,2d+1)$. We show that there exists a cubillage $\mathcal{Q}_{\mathcal{T}}$ of $Z(n,2d+2)$ such that $\mathrm{Sp}(\mathcal{Q}_{\mathcal{T}}) \supseteq \Sigma(\mathcal{T})$. Consider the triangulation $\hat{\mathcal{T}}$ of $C(n+1,2d+2)$ defined in \cite[Definition 4.1]{rambau}. By Theorem~\ref{thm-g-even}, there is a cubillage $\mathcal{Q}'$ of $Z(n+1,2d+3)$ such that $g(\mathcal{Q}')=\hat{\mathcal{T}}$. By definition of $\hat{\mathcal{T}}$, we have that $\Sigma(\mathcal{T})\cup\Sigma(\mathcal{T})\ast(n+1) \subseteq \Sigma(\hat{\mathcal{T}}) \subseteq \mathrm{Sp}(\mathcal{Q}')$. By \cite[Lemma 5.2]{dkk}, if we take the $(n + 1)$-contraction of $\mathcal{Q}'$ then we get a membrane $\mathcal{M}$ in $\mathcal{Q}'/(n + 1)$ as the image of the $(n + 1)$-pie, and we have that $\mathrm{Sp}(\mathcal{M}) \supseteq \Sigma(\mathcal{T})$. We therefore define $\mathcal{Q}_{\mathcal{T}} = \mathcal{M}$, recalling that $\mathcal{M}$ is a cubillage of $Z(n, 2d + 2)$. By Lemma~\ref{lem-all-simps}, we must have that $g(\mathcal{Q}_{\mathcal{T}})=\mathcal{T}$.
\end{proof}

\begin{corollary}\label{cor:g_surj}
The map $g \colon \mathcal{B}(n, \delta + 1) \to \mathcal{S}(n, \delta)$ is a surjection.
\end{corollary}

\begin{remark}
In \cite[Theorem 4.10]{kv-poly}, Kapranov and Voevodsky gave a map $f\colon \mathcal{B}(n,\delta) \rightarrow \mathcal{S}(n+2,\delta+1)$ which they stated was a surjection. A proof of this statement remains unfound. It was shown in \cite[Proposition 7.1]{thomas-bst} that there is a factorisation \[
\begin{tikzcd}
\mathcal{B}(n,\delta) \ar[dr,"f"] \ar[rr,"\overline{g}"] && \mathcal{S}(n, \delta-1), \\
& \mathcal{S}(n+2,\delta+1) \ar[ru, dashed, two heads] &
\end{tikzcd}\] where $\overline{g}$ is the dual map to $g$ from Remark~\ref{rmk:g_dual} and the dotted map is a surjection by \cite[Corollary 4.3]{rambau}.

The map $f$ should not only be a surjection, but also a quotient map of posets, as we show  is true of the map $g$ in this paper. This was shown for $\delta = 1$ by Reading \cite{reading_cambrian}, drawing upon \cite{bw_shell_2}. However, note that $f$ cannot in general realise $\mathcal{S}(n + 2, \delta + 1)$ as a quotient of $\mathcal{B}(n, \delta)$ by an order congruence in the sense used in \cite{reading_cambrian}. This is because the equivalence classes of an order congruence must be intervals, but \cite[Section 6]{thomas-bst} shows that the fibres of the map $f$ are not always intervals. Hence $f$ can only be a quotient map of posets in a more general sense, such as that considered in this paper.
\end{remark}

\section{Fullness}\label{sect-quot}

We now show that the map $g$ is full, and hence is a quotient map of posets. To do this, we must show that if $\mathcal{T} \leqslant \mathcal{T}'$ for triangulations $\mathcal{T}, \mathcal{T}'$ of $C(n,\delta)$, then there are cubillages $\mathcal{Q}, \mathcal{Q}'$ of $Z(n,\delta+1)$ such that $g(\mathcal{Q})=\mathcal{T}$, $g(\mathcal{Q}')=\mathcal{T}'$, and $\mathcal{Q} \leqslant \mathcal{Q}'$. We follow the approach of Section~\ref{sect-surj}, whereby we work explicitly for even-dimensional triangulations, and then use this to show the result for odd dimensions. Indeed, we show that for triangulations $\mathcal{T}, \mathcal{T}'$ of $C(n, 2d)$ with $\mathcal{T} \leqslant \mathcal{T}'$, we have $\mathcal{Q}_{\mathcal{T}} \leqslant \mathcal{Q}_{\mathcal{T}'}$. For this, it suffices to show that if $\mathcal{T} \lessdot \mathcal{T}'$, then $\mathcal{Q}_{\mathcal{T}} < \mathcal{Q}_{\mathcal{T}'}$. To do this, we find a sequence of increasing flips from $\mathcal{Q}_{\mathcal{T}}$ to $\mathcal{Q}_{\mathcal{T}'}$.

We wish to continue working in the framework of separated collections, as in Section~\ref{sect-surj}. Hence, we must show what the covering relations of the higher Bruhat orders are in this framework.

\begin{theorem}\label{thm-comb-flips}
Given cubillages $\mathcal{Q}, \mathcal{Q}'$ of $Z(n, \delta + 1)$ we have that $\mathcal{Q} \lessdot \mathcal{Q}'$ if and only if $\mathrm{Sp}(\mathcal{Q}') = (\mathrm{Sp}(\mathcal{Q})\setminus \{A\})\cup \{B\}$, where $A$ $\delta$-interweaves $B$. Moreover, in this case $A$ tightly $\delta$-interweaves $B$.
\end{theorem}
\begin{proof}
The forwards direction follows from \cite[Proposition 8.1]{dkk-survey}. Namely, if the increasing flip from $\mathcal{Q}$ to $\mathcal{Q}'$ is induced by the face $\Gamma$ of $Z(n, n)$, then $\Gamma$ has a vertex $\xi_{A}$ and a vertex $\xi_{B}$ such that $A$ tightly $\delta$-interweaves $B$, $\pi_{n, \delta + 2}(\xi_{A})$ is only contained in the lower facets of $\pi_{n, \delta + 2}(\Gamma)$, $\pi_{n, \delta + 2}(\xi_{B})$ is only contained in the upper facets of $\pi_{n, \delta + 2}(\Gamma)$, and every other vertex of $\pi_{n, \delta + 2}(\Gamma)$ is contained in at least one lower facet and at least one upper facet. Hence, $\mathrm{Sp}(\mathcal{Q}') = (\mathrm{Sp}(\mathcal{Q})\setminus \{A\})\cup \{B\}$, where $A$ tightly $\delta$-interweaves $B$.

We now prove the backwards direction, supposing that $\mathrm{Sp}(\mathcal{Q}') = (\mathrm{Sp}(\mathcal{Q})\setminus \{A\})\cup \{B\}$, where $A$ $\delta$-interweaves $B$. Let $A' \subseteq A \setminus B$ and $B' \subseteq B\setminus A$ witness the fact that $A$ $\delta$-interweaves $B$.

We consider first the case where $\delta=2d$. We begin by proving that $A' = A \setminus B$ and $B' = B \setminus A$, so that $A$ tightly $2d$-interweaves $B$. The vertex $\xi_{A}$ must be an internal vertex in the cubillage $\mathcal{Q}$, since subsets corresponding to boundary vertices are contained in every $2d$-separated collection. Therefore, $\xi_{A}$ must be a vertex of at least two cubes in $\mathcal{Q}$, and so must have at least $2d+2$ edges emanating from it. The subsets at the other end of each of these edges must be $2d$-separated from $B$, so the edges must either add elements of $B'$ or remove elements of $A'$. Since $\# A' \cup B' = 2d + 2$, the edges emanating from $\xi_{A}$ in $\mathcal{Q}$ must be precisely the edges which remove elements of $A'$ and add elements of $B'$. Now suppose that there exists $a \in A \setminus (A' \cup B)$. Then $a \in (b'_{i-1}, b'_{i})$ for some $i \in \mathbb{Z}/(d + 1)\mathbb{Z}$. But this implies that $A \setminus \{a'_{i}\}$ $\delta$-interweaves $B$, which contradicts the fact that the edge from $\xi_{A}$ to $\xi_{A \setminus \{a'_{i}\}}$ is in the cubillage $\mathcal{Q}$. Hence $A' = A \setminus B$. The argument that $B' = B \setminus A$ is similar.

Therefore $\xi_{A}$ is incident to $2d+2$ edges in the cubillage, where $d+1$ of the edges add elements of $B'$ and $d+1$ of the edges remove elements of $A'$. The cubes with $\xi_{A}$ as a vertex are generated by a choice of $2d+1$ of these edges. If $\mathcal{P}$ is the union of cubes in $\mathcal{Q}$ with $\xi_{A}$ as a vertex, then $\mathcal{P}$ is a set of facets of a $(2d+2)$-face $\Gamma$ of $Z(n, n)$ which has initial vertex $\xi_{A \cap B}$ and which is generated by $A' \cup B'$. By \cite[Proposition 8.1]{dkk-survey}, $\pi_{n, \delta + 2}(\mathcal{P})$ gives the lower facets of $\pi_{n, \delta + 2}(\Gamma)$, since $\pi_{n, \delta + 2}(\mathcal{P})$ consists of all the facets of $\pi_{n, \delta + 2}(\Gamma)$ which contain $\pi_{n, \delta + 2}(\xi_{A})$. Since, likewise, the upper facets of $\pi_{n, \delta + 2}(\Gamma)$ are precisely those containing $\pi_{n, \delta + 2}(\xi_{B})$, we obtain that $\mathcal{Q}'$ is an increasing flip of $\mathcal{Q}$.

For $\delta=2d+1$, the argument is similar. We deduce that $\xi_{A}$ has $2d + 3$ edges emanating from it in $\mathcal{Q}$, $d + 1$ of which remove elements of $A'$ and $d + 2$ of which add elements of $B'$. To show that $A' = A \setminus B$ and $B' = B \setminus A$, the only extra thing to consider is the possibility that we have $a \in A \setminus (A' \cup B)$ such that $a < b'_{0}$ or $a > b'_{d+1}$. But in the first instance here, we have that $B$ $\delta$-interweaves $A \cup \{b'_{d+1}\}$, since \[a < b'_{0} < a'_{0} < b'_{1} < \dots < b'_{d} < a'_{d}.\] But this is a contradiction, since we know that the edge from $\xi_{A}$ to $\xi_{A \cup \{b'_{d+1}\}}$ is in $\mathcal{Q}$. In the second instance, we have that $B$ $\delta$-interweaves $A \cup \{b'_{0}\}$, when we know that the edge from $\xi_{A}$ to $\xi_{A \cup \{b'_{0}\}}$ is in $\mathcal{Q}$. The remainder of the case where $\delta = 2d - 1$ is similar.
\end{proof}

In the setting of the above theorem, we say that $(A, B)$ is the \emph{exchange pair} of the flip and that we \emph{exchange} $A$ for $B$. Using this characterisation of increasing flips, it can be seen that, in order to show that $\mathcal{Q}_{\mathcal{T}} \leqslant \mathcal{Q}_{\mathcal{T}'}$, we must show that we can gradually exchange the elements of $\mathrm{Sp}(\mathcal{Q}_{\mathcal{T}})\setminus\mathrm{Sp}(\mathcal{Q}_{\mathcal{T}'})$ for the elements of $\mathrm{Sp}(\mathcal{Q}_{\mathcal{T}'}) \setminus \mathrm{Sp}(\mathcal{Q}_{\mathcal{T}})$. If $|S|$ is the simplex inducing the increasing flip from $\mathcal{T}$ to $\mathcal{T}'$, then $\mathrm{Sp}(\mathcal{Q}_{\mathcal{T}})\setminus\mathrm{Sp}(\mathcal{Q}_{\mathcal{T}'}) = \mathcal{I}_{l}(S, n)$ and $\mathrm{Sp}(\mathcal{Q}_{\mathcal{T}'}) \setminus \mathrm{Sp}(\mathcal{Q}_{\mathcal{T}}) = \mathcal{I}_{u}(S, n)$, as in Lemma~\ref{lem:right_size}. Hence, we will define a sequence of exchanges which replaces $\mathcal{I}_{l}(S, n)$ with $\mathcal{I}_{u}(S, n)$. To show that our sequence of exchanges works, we will need the following lemma.

\begin{lemma}\label{lem:interweaving_criterion_for_flips}
Let \[I = [s_{0}, s^{i}_{0}] \cup [s_{1}, s^{i}_{1}] \cup \dots \cup [s_{2d}, s^{i}_{2d}] \cup [s_{2d+1}, s^{i}_{2d+1}]\] and \[J = [s_{0}, s^{j}_{0}] \cup [s_{1}, s^{j}_{1}] \cup \dots \cup [s_{2d}, s^{j}_{2d}] \cup [s_{2d+1}, s^{j}_{2d+1}].\] Then $I$ $2d$-interweaves $J$ if and only if, for all $r$, \[s^{j}_{2r} < s^{i}_{2r} \text{ and } s^{i}_{2r+1} < s^{j}_{2r+1}.\]
\end{lemma}
\begin{proof}
If we have that, for all $r$, $s^{j}_{2r} < s^{i}_{2r}$ and $s^{i}_{2r+1} < s^{j}_{2r+1}$, then we have that $\{s^{i}_{0}, s^{i}_{2}, \dots, s^{i}_{2d}\} \subseteq I \setminus J$ and $\{s_{1}^{j}, s_{3}^{j}, \dots, s_{2d+1}^{j}\} \subseteq J \setminus I$ with \[s_{0}^{i} < s_{1}^{j} < s_{2}^{i} < s_{3}^{j} < \dots < s_{2d}^{i} < s_{2d+1}^{j}.\] Hence $I$ $2d$-interweaves $J$.

Conversely, suppose that $I$ $2d$-interweaves $J$, and let $X \subseteq I \setminus J$ and $Y \subseteq J \setminus I$ witness this. We cannot have both $x_{p}, x_{q} \in [s_{t}, s^{i}_{t}]$ for $p \neq q$, since this implies that $y_{r} \in [s_{t}, s^{i}_{t}]$ for $p \leqslant r < q$. Furthermore, we cannot have both $x_{p} \in [s_{t}, s_{t}^{i}]$ and $y_{p} \in [s_{t}, s_{t}^{j}]$, since we must have either $[s_{t}, s_{t}^{i}] \subseteq [s_{t}, s_{t}^{j}]$ or $[s_{t}, s_{t}^{j}] \subseteq [s_{t}, s_{t}^{i}]$. By the pigeonhole principle and the fact that $x_{0} < y_{0}$, we deduce that $x_{r} \in [s_{2r}, s_{2r}^{i}]$ and $y_{r} \in [s_{2r + 1}, s_{2r + 1}^{j}]$ for all $r$. But this implies that $s^{j}_{2r} < s^{i}_{2r}$ and $s^{i}_{2r + 1} < s^{j}_{2r + 1}$ for all $r$.
\end{proof}

It is now useful for us to obtain an explicit map for the bijection from Lemma~\ref{lem-card-flip}. This allows us to construct the sequence of exchanges which replaces $\mathcal{I}_{l}(S, n)$ with $\mathcal{I}_{u}(S, n)$.

\begin{construction}\label{constr:bij}
Given $S \in \binom{[n]}{2d + 2}$, we define
\begin{align*}
\mathcal{I}(S, n) &= \mathcal{I}_{l}(S, n) \cup \mathcal{I}_{u}(S, n), \\
\mathcal{I}'(S, n) &= \mathcal{I}'_{l}(S, n) \cup \mathcal{I}'_{u}(S, n).
\end{align*}
In order to get a convenient parametrisation of these sets, we define a map
\begin{align*}
\phi \colon \prod_{i \in \mathbb{Z}/(2d + 2)\mathbb{Z}} [0, s_{i + 1} - s_{i}] &\rightarrow 2^{[n]} \\
(n_{0}, n_{1}, \dots, n_{2d + 1}) &\mapsto \bigcup_{i \in \mathbb{Z}/(2d + 2)\mathbb{Z}} [s_{i}, s_{i} + n_{i} - 1].
\end{align*}
We abbreviate $\mathbf{n} = (n_{0}, n_{1}, \dots, n_{2d + 1})$. Then
\begin{itemize}
\item $\phi(\mathbf{n}) \in \mathcal{I}'_{l}(S,n)$ if and only if $n_{2i - 1} < s_{2i} - s_{2i - 1}$ and $n_{2i} > 0$ for all $i \in \mathbb{Z}/(d+1)\mathbb{Z}$;
\item $\phi(\mathbf{n}) \in \mathcal{I}'_{u}(S,n)$ if and only if $n_{2i} < s_{2i+1} - s_{2i}$ and $n_{2i+1} > 0$ for all $i \in \mathbb{Z}/(d+1)\mathbb{Z}$;
\item $\phi(\mathbf{n}) \in \mathcal{I}_{l}(S,n)$ if and only if $n_{2i-1} < s_{2i} - s_{2i-1}$ and $n_{2i} > 0$ for all $i \in \mathbb{Z}/(d+1)\mathbb{Z}$, and there exists a $j \in \mathbb{Z}/(d+1)\mathbb{Z}$ such that either $n_{2j+1} = 0$ or $n_{2j} = s_{2j+1} - s_{2j}$;
\item $\phi(\mathbf{n}) \in \mathcal{I}_{u}(S,n)$ if and only if $n_{2i} < s_{2i+1} - s_{2i}$ and $n_{2i+1} > 0$ for all $i \in \mathbb{Z}/(d+1)\mathbb{Z}$, and there exists a $j \in \mathbb{Z}/(d+1)\mathbb{Z}$ such that either $n_{2j} = 0$, or $n_{2j-1} = s_{2j} - s_{2j-1}$.
\end{itemize}
We then obtain an explicit bijection by defining a map \[\psi \colon \mathcal{I}_{l}(S,n) \rightarrow \mathcal{I}_{u}(S,n)\] as follows. Let $I \in \mathcal{I}_{l}(S,n)$ such that $I = \phi(\mathbf{n})$ and let $\mathbf{t} = (-1, 1, -1, 1, \dots, -1, 1)$. Further, define \[\lambda_{I} = \max\left\lbrace\lambda \in \mathbb{Z}_{>0} \mathrel{\Big|} \mathbf{n} + \lambda\mathbf{t} \in \prod_{i \in \mathbb{Z}/(2d + 2)\mathbb{Z}} [0, s_{i + 1} - s_{i}]\right\rbrace.\] By construction, \[\phi(\mathbf{n} + \lambda_{I}\mathbf{t}) \in \mathcal{I}_{u}(S, n),\] since we must either have some $j \in \mathbb{Z}/(d+1)\mathbb{Z}$ such that $s_{2j} - \lambda_{I} = 0$, or some $j \in \mathbb{Z}/(d+1)\mathbb{Z}$ such that $s_{2j-1} + \lambda_{I} = s_{2j} - s_{2j-1}$, otherwise $\lambda_{I}$ would not be maximal. Therefore define \[\psi(I) = \phi(\mathbf{n} + \lambda_{I}\mathbf{t}).\]

It can be seen that the map $\psi$ is a bijection because one may define its inverse as follows. Let $J \in \mathcal{I}_{u}(S,n)$ such that $J = \phi(\mathbf{n})$. Then let \[\mu_{J} = \max\left\lbrace\mu \in \mathbb{Z}_{>0} \mathrel{\Big|} \mathbf{n} - \mu\mathbf{t} \in \prod_{i \in \mathbb{Z}/(2d + 2)\mathbb{Z}} [0, s_{i + 1} - s_{i}]\right\rbrace.\] By construction, \[\phi(\mathbf{n} - \mu_{J}\mathbf{t}) \in \mathcal{I}_{l}(S, n),\] since we must either have some $j \in \mathbb{Z}/(d+1)\mathbb{Z}$ such that $n_{2j+1} - \mu_{J} = 0$, or some $j \in \mathbb{Z}/(d+1)\mathbb{Z}$ such that $n_{2j} + \mu_{J} = s_{2j+1} - s_{2j}$. It is then clear that \[\psi^{-1}(J) = \phi(\mathbf{n} - \mu_{J}\mathbf{t}).\]
\end{construction}

\begin{theorem}\label{thm-even-quot}
Given triangulations $\mathcal{T}, \mathcal{T}'$ of $C(n,2d)$ such that $\mathcal{T} \lessdot \mathcal{T}'$, there exist cubillages $\mathcal{Q}_{0}, \dots, \mathcal{Q}_{r}$ of $Z(n,2d+1)$ such that $\mathcal{Q}_{0} = \mathcal{Q}_{\mathcal{T}}$, $\mathcal{Q}_{r} = \mathcal{Q}_{\mathcal{T}'}$ and \[\mathcal{Q}_{0} \lessdot \mathcal{Q}_{1} \lessdot \dots \lessdot \mathcal{Q}_{r},\] so that $\mathcal{Q}_{\mathcal{T}} \leqslant \mathcal{Q}_{\mathcal{T}'}$.
\end{theorem}
\begin{proof}
Suppose that the increasing flip of $\mathcal{T}$ which gives $\mathcal{T}'$ is induced by the $(2d+1)$-face $|S|$ of $C(n, n - 1)$. Then $\intsp(\mathcal{Q}_{\mathcal{T}}) \setminus \intsp(\mathcal{Q}_{\mathcal{T}'}) = \mathcal{I}_{l}(S,n)$ and $\intsp(\mathcal{Q}_{\mathcal{T}'}) \setminus \intsp(\mathcal{Q}_{\mathcal{T}}) = \mathcal{I}_{u}(S,n)$. Let $\mathcal{R} = \intsp(\mathcal{Q}_{\mathcal{T}}) \setminus \mathcal{I}_{l}(S,n) = \intsp(\mathcal{Q}_{\mathcal{T}'}) \setminus \mathcal{I}_{u}(S,n)$. Hence we must find a sequence of flips starting at $\mathcal{Q}_{\mathcal{T}}$ which gradually replaces $\mathcal{I}_{l}(S, n)$ with $\mathcal{I}_{u}(S, n)$.

The flips of cubillages we wish to perform are as follows. Given $\phi(\mathbf{n}) \in \mathcal{I}_{l}(S, n)$, we make the sequence of exchanges \[\phi(\mathbf{n}) \leadsto \phi(\mathbf{n} + \mathbf{t}) \leadsto \dots \leadsto \phi(\mathbf{n} + (\lambda_{\phi(\mathbf{n})} - 1)\mathbf{t}) \leadsto \phi(\mathbf{n} + \lambda_{\phi(\mathbf{n})}\mathbf{t)},\] where $\phi(\mathbf{n}) \leadsto \phi(\mathbf{n} + \mathbf{t})$ means that we remove $\phi(\mathbf{n})$ and replace it with $\phi(\mathbf{n} + \mathbf{t})$. Hence the set of exchange pairs in our sequence of flips from $\mathcal{Q}_{\mathcal{T}}$ to $\mathcal{Q}_{\mathcal{T}'}$ is \[\{(\phi(\mathbf{n} + r\mathbf{t}), \phi(\mathbf{n} + (r + 1)\mathbf{t}) \mid \phi(\mathbf{n}) \in \mathcal{I}_{l}(S, n), ~ 0 \leqslant r < \lambda_{\phi(\mathbf{n})}\}.\] We must show that there is an order in which we can make these exchanges such that after each exchange we still have a $2d$-separated collection. Here each exchange gives an increasing flip by Theorem~\ref{thm-comb-flips}. Note further that $\phi(\mathbf{n} + r\mathbf{t})$ and $\phi(\mathbf{n} + (r + 1)\mathbf{t})$ are tightly $2d$-interweaving, as we know must be the case from Theorem~\ref{thm-comb-flips}.

Our exchanges give a bijection
\begin{align*}
\mathcal{I}'(S, n) \setminus \mathcal{I}_{u}(S, n) &\rightarrow \mathcal{I}'(S, n) \setminus \mathcal{I}_{l}(S, n) \\
\phi(\mathbf{n}) &\mapsto \phi(\mathbf{n} + \mathbf{t}).
\end{align*}
Hence, we have one exchange per element of $\mathcal{I}'(S, n) \setminus \mathcal{I}_{u}(S, n)$. By Construction~\ref{constr:bij}, we have that $\phi$ is a bijection between $[1, s_{1} - s_{0}] \times [0, s_{2} - s_{1} - 1] \times \dots \times [1, s_{2d+1} - s_{2d}] \times [0, s_{0} - s_{2d+1} - 1 + n]$ and $\mathcal{I}'(S, n) \setminus \mathcal{I}_{u}(S, n)$. The set $[1, s_{1} - s_{0}] \times [0, s_{2} - s_{1} - 1] \times \dots \times [1, s_{2d+1} - s_{2d}] \times [0, s_{0} - s_{2d+1} - 1 + n]$ is a lattice under the order given by \[(n_{0}, n_{1}, \dots, n_{2d+1}) \leqslant (n'_{0}, n'_{1}, \dots, n'_{2d+1})\] if and only if for all $j$ \[n'_{2j} \leqslant n_{2j} \text{ and } n_{2j+1} \leqslant n'_{2j+1},\] since this is just the usual product order, but reversed on coordinates with even index.

We claim that any linear extension $\mathbf{n}^{1} < \dots < \mathbf{n}^{r}$ of this lattice gives an order on $\mathcal{I}'(S, n)\setminus \mathcal{I}_{u}(S, n)$ such that if $\mathcal{C}_{0} := \mathrm{Sp}(\mathcal{Q}_{\mathcal{T}})$ and $\mathcal{C}_{i} := (\mathcal{C}_{i - 1} \setminus \{\phi(\mathbf{n}^{i})\}) \cup \{\phi(\mathbf{n}^{i} + \mathbf{t})\}$, then $\mathcal{C}_{i}$ is $2d$-separated for all $i$. Note first that we always must have $\phi(\mathbf{n}^{i}) \in \mathcal{C}_{i - 1}$. This is because either $\phi(\mathbf{n}^{i}) \in \mathcal{I}_{l}(S, n)$ or $\phi(\mathbf{n}^{i} - \mathbf{t}) \in \mathcal{I}'(S, n) \setminus \mathcal{I}_{u}(S, n)$. Hence, either $\phi(\mathbf{n}^{i}) \in \mathcal{C}_{0}$, or $\phi(\mathbf{n}^{i})$ is the result of an earlier exchange, since $\mathbf{n}^{i} - \mathbf{t} < \mathbf{n}^{i}$ in our order.

Now suppose that $\mathcal{C}_{i}$ is not $2d$-separated for some $i$. We may choose the minimal $i$ for which this is the case. We first show that no element of $\mathcal{I}'(S, n)$ is $2d$-interweaving with any element of $\mathcal{R}$. Suppose, on the contrary, that there exist $I \in \mathcal{I}'(S, n)$ and $J \in \mathcal{R}$ such that $I$ and $J$ are $2d$-interweaving. Then, by Lemma~\ref{lem-int-endpts}, we have $X \subseteq \widehat{I}$ and $Y \subseteq \widehat{J}$ such that $\# X = \# Y = d + 1$ and $X$ and $Y$ are $2d$-interweaving. We have that $X \subseteq \widehat{I} \subseteq S$, and since $\# X = d + 1$, we must have either $X \not\supseteq S_{u}$, or $X \not\supseteq S_{l}$. If $X \not\supseteq S_{u}$, then $X \subseteq F$ for a $2d$-simplex $|F|$ of $\mathcal{T}$, by Gale's Evenness Criterion. This gives a contradiction, since $|F|$ and $|\widehat{J}|$ are both simplices of $\mathcal{T}$ and $(X, Y)$ is a circuit. One can derive a similar contradiction using $\mathcal{T}'$ when $X \not\supseteq S_{l}$.

Therefore, if $\mathcal{C}_{i}$ is not $2d$-separated, it must be because $\phi(\mathbf{n}^{i} + \mathbf{t})$ is $2d$-interweaving with an element $I \in \mathcal{I}(S, n) \cap \mathcal{C}_{i}$. By Lemma~\ref{lem:interweaving_criterion_for_flips}, we must have \[I = [s_{0}, s'_{0}] \cup [s_{1}, s'_{1}] \cup \dots \cup [s_{2d+1}, s'_{2d+1}] \in \mathcal{C}_{i} \setminus \{\phi(\mathbf{n}^{i} + \mathbf{t})\} = \mathcal{C}_{i - 1} \setminus \{\phi(\mathbf{n}^{i})\}\] such that either $s_{2j} + (n_{2j}^{i} - 1) - 1 < s'_{2j}$ and $s'_{2j+1} < s_{2j+1} + (n_{2j+1}^{i} + 1) - 1$ for all $j$, or $s'_{2j} < s_{2j} + (n_{2j}^{i} - 1) - 1$ and $s_{2j+1} + (n_{2j+1}^{i} + 1) - 1 < s'_{2j+1}$ for all $j$. In the latter case, we also have that $s'_{2j} < s_{2j} + n_{2j}^{i} - 1$ and $s_{2j+1} + n_{2j+1}^{i} - 1 < s'_{2j+1}$, so that $\phi(\mathbf{n}^{i})$ also $2d$-interweaves $I$, which means that $\mathcal{C}_{i - 1}$ is not $2d$-separated. This contradicts $i$ being the minimal index such that this was the case. In the former case, we have that $I$ precedes $\phi(\mathbf{n}^{i})$ in our chosen order on $\mathcal{I}'(S, n)\setminus \mathcal{I}_{u}(S, n)$. This means that $I$ must have already been exchanged, which is also a contradiction.

Therefore, we have cubillages $\mathcal{Q}_{0}, \dots, \mathcal{Q}_{r}$ such that $\mathcal{C}_{i} = \mathrm{Sp}(\mathcal{Q}_{i})$ for each $i$. By Theorem~\ref{thm-comb-flips}, we have \[\mathcal{Q}_{0} \lessdot \mathcal{Q}_{1} \lessdot \dots \lessdot \mathcal{Q}_{r}.\] By construction, we have that $\mathcal{Q}_{0} = \mathcal{Q}_{\mathcal{T}}$ and $\mathcal{Q}_{r} = \mathcal{Q}_{\mathcal{T}'}$.
\end{proof}

\begin{example}\label{ex:even_quot}
We give an example of the construction used to prove Theorem~\ref{thm-even-quot}.

\begin{enumerate}[wide]
\item Consider the triangulation $\mathcal{T}$ of the heptagon $C(7, 2)$ given by $\simp(\mathcal{T}) = \{13, 16, 35, 36\}$. We perform the increasing flip on this triangulation induced by the simplex $|1236|$, thereby obtaining the triangulation $\mathcal{T}'$ of $C(7, 2)$ with $\simp(\mathcal{T}') = \{16, 26, 35, 36\}$.

We have
\begin{align*}
\intsp(\mathcal{Q}_{\mathcal{T}}) = \{&13, 16, 35, 36,\\
&126, 134, 136, 346, 356, 367,\\
&1236, 1345, 1346, 1367, 3467, 3567, \\
&12346, 13456, 13467, 13567\}
\end{align*}
and
\begin{align*}
\intsp(\mathcal{Q}_{\mathcal{T}'}) = \{&16, 26, 35, 36\\
&126, 236, 267, 346, 356, 367,\\
&1236, 1367, 2346, 2367, 3467, 3567,\\
&12346, 13467, 13567, 23467\}.
\end{align*}
Moreover, \[\intsp(\mathcal{Q}_{\mathcal{T}})\setminus \intsp(\mathcal{Q}_{\mathcal{T}'}) = \mathcal{I}_{l}(1236, 7) = \{13, 134, 136, 1345, 1346, 13456\}\] and \[\intsp(\mathcal{Q}_{\mathcal{T}'})\setminus \intsp(\mathcal{Q}_{\mathcal{T}}) = \mathcal{I}_{u}(1236, 7) = \{26, 236, 267, 2346, 2367, 23467\}.\] We illustrate how we can gradually replace elements of $\mathcal{I}_{l}(1236, 7)$ in $\intsp(\mathcal{Q}_{\mathcal{T}})$ with the elements of $\mathcal{I}_{u}(1236, 7)$, whilst ensuring that the collection remains $2$-separated.

The coordinate parameterisation of $\mathcal{I}'(1236, 7)$ by $\phi$ gives
\begin{align*}
\phi(1, 0, 1, 0) &= 13, \\
\phi(1, 0, 2, 0) &= 134, \\
\phi(1, 0, 1, 1) &= 136, \\
\phi(1, 0, 3, 0) &= 1345, \\
\phi(1, 0, 2, 1) &= 1346, \\
\phi(1, 0, 3, 1) &= 13456, \\
\phi(0, 1, 0, 1) &= 26, \\
\phi(0, 1, 1, 1) &= 236, \\
\phi(0, 1, 0, 2) &= 267, \\
\phi(0, 1, 2, 1) &= 2346, \\
\phi(0, 1, 1, 2) &= 2367, \\
\phi(0, 1, 2, 2) &= 23467. 
\end{align*}
The bijection $\psi\colon \mathcal{I}_{l}(1236, 7) \rightarrow \mathcal{I}_{u}(1236, 7)$ in this case gives
\[
\begin{tabular}{rcccccl}
13 & = & $\phi(1, 0, 1, 0)$ & $\mapsto$ & $\phi(0, 1, 0, 1)$ & = & 26,\\
134 & = & $\phi(1, 0, 2, 0)$ & $\mapsto$ & $\phi(0, 1, 1, 1)$ & = & 236, \\
136 & = & $\phi(1, 0, 1, 1)$ & $\mapsto$ & $\phi(0, 1, 0, 2)$ & = & 267, \\
1345 & = & $\phi(1, 0, 3, 0)$ & $\mapsto$ & $\phi(0, 1, 2, 1)$ & = & 2346, \\
1346 & = & $\phi(1, 0, 2, 1)$ & $\mapsto$ & $\phi(0, 1, 1, 2)$ & = & 2367, \\
13456 & = & $\phi(1, 0, 3, 1)$ & $\mapsto$ & $\phi(0, 1, 2, 2)$ & = & 23467. \\
\end{tabular}
\]
Note that in this example, we have that $\mathcal{I}'_{l}(1236, 7) = \mathcal{I}_{l}(1236, 7)$ and $\mathcal{I}'_{u}(1236, 7) = \mathcal{I}_{u}(1236, 7)$, since we cannot have $\widehat{I} = 1236$ for any subset $I$. Thus we consider the lattice on $\mathcal{I}'(1237, 6)\setminus \mathcal{I}_{u}(1237, 6) = \mathcal{I}_{l}(1236, 7)$ given by
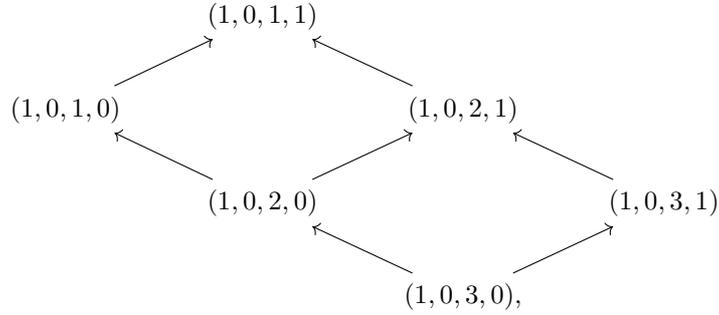
\[
\begin{tikzcd}
& (1, 0, 1, 1) && \\
(1, 0, 1, 0) \ar[ur] && (1, 0, 2, 1) \ar[ul] & \\
& (1, 0, 2, 0) \ar[ul] \ar[ur] && (1, 0, 3, 1) \ar[ul] \\
&& (1, 0, 3, 0) \ar[ul] \ar[ur],
\end{tikzcd}
\]
which is
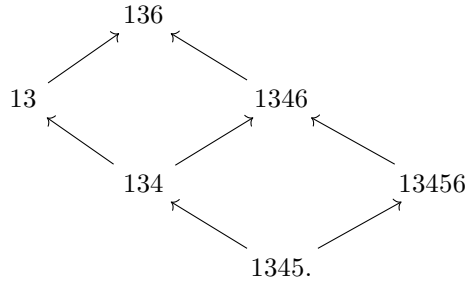
\[
\begin{tikzcd}
& 136 && \\
13 \ar[ur] && 1346 \ar[ul] & \\
& 134 \ar[ul] \ar[ur] && 13456 \ar[ul] \\
&& 1345. \ar[ul] \ar[ur]
\end{tikzcd}
\]
Note here that we place minimal element of the lattice at the bottom. Therefore, by Theorem~\ref{thm-even-quot}, we may perform the exchanges replacing $\phi(\mathbf{n})$ by $\phi(\mathbf{n} + \mathbf{t})$ in an order given by any linear extension of
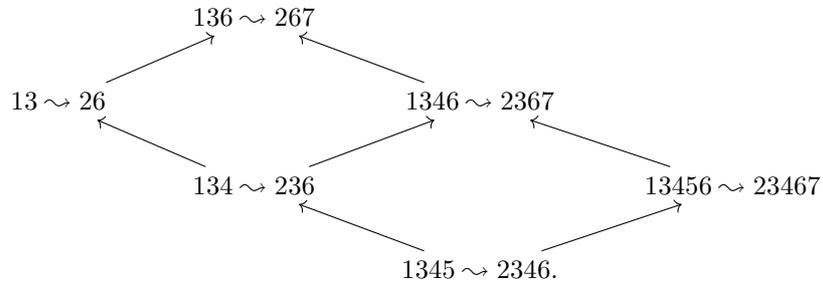
\[
\begin{tikzcd}
& 136 \leadsto 267 && \\
13 \leadsto 26 \ar[ur] && 1346 \leadsto 2367 \ar[ul] & \\
& 134 \leadsto 236 \ar[ul] \ar[ur] && 13456 \leadsto 23467 \ar[ul] \\
&& 1345 \leadsto 2346. \ar[ul] \ar[ur]
\end{tikzcd}
\]
Note here that we first make the exchange at the bottom of the lattice, and then move up.

\item We now give an example where we do not have $\mathcal{I}(S, n) = \mathcal{I}'(S, n)$. This example is somewhat larger than the previous example, so we do not go through it in the same level of detail.

Indeed, we do not consider full triangulations, but only the set $\mathcal{I}_{l}(1357, 8)$, which we wish to replace with the set $\mathcal{I}_{u}(1357, 8)$. Here we have $\mathcal{I}'_{l}(1357, 8) = \mathcal{I}_{l}(1357, 8) \cup \{1357\}$ and $\mathcal{I}'_{u}(1357, 8) = \mathcal{I}_{u}(1357, 8) \cup \{1357\}$. The sequence of exchanges from $\mathcal{I}_{l}(1357, 8)$ to $\mathcal{I}_{u}(1357, 8)$ is given by the bijection $\phi(\mathbf{n}) \mapsto \phi(\mathbf{n} + \mathbf{t})$ from $\mathcal{I}'(1357, 8)\setminus\mathcal{I}_{u}(1357, 8)$ to $\mathcal{I}'(1357, 8)\setminus\mathcal{I}_{l}(1357, 8)$.

Any sequence of exchanges done in the order of any linear extension of the following lattice will preserve 2-separatedness. One can check that this is the lattice from the proof of Theorem~\ref{thm-even-quot}.
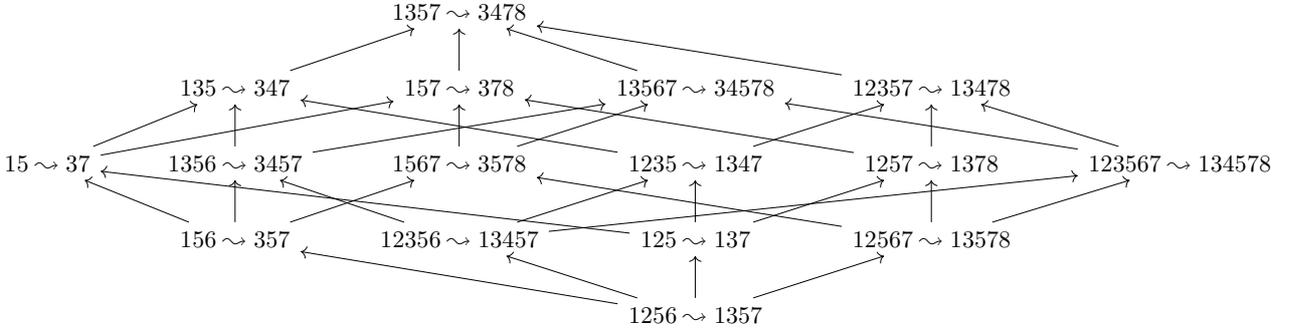
\[
\adjustbox{scale=0.9, center}{
\begin{tikzcd}
&& 1357 \leadsto 3478 &&& \\
& 135 \leadsto 347 \ar[ur] & 157 \leadsto 378 \ar[u] & 13567 \leadsto 34578 \ar[ul] & 12357 \leadsto 13478 \ar[ull] & \\
15 \leadsto 37 \ar[ur] \ar[urr] & 1356 \leadsto 3457 \ar[u] \ar[urr] & 1567 \leadsto 3578 \ar[u] \ar[ur] & 1235 \leadsto 1347 \ar[ull] \ar[ur] & 1257 \leadsto 1378 \ar[ull] \ar[u] & 123567 \leadsto 134578 \ar[ull] \ar[ul] \\
& 156 \leadsto 357 \ar[ul] \ar[u] \ar[ur] & 12356 \leadsto 13457 \ar[ul] \ar[ur] \ar[urrr] & 125 \leadsto 137 \ar[ulll] \ar[u] \ar[ur] & 12567 \leadsto 13578 \ar[ull] \ar[u] \ar[ur] & \\
&&& 1256 \leadsto 1357 \ar[ull] \ar[ul] \ar[u] \ar[ur] &&
\end{tikzcd}
}
\]
Note that here, since $1357 \in \mathcal{I}'(1357, 8)\setminus\mathcal{I}_{l}(1357, 8)$, but $1357 \notin \mathcal{I}_{u}(1357, 8)$, we have that $1256 \leadsto 1357 \leadsto 3478$. That is, $1357$ is only an intermediate subset in the sequence of exchanges from $\mathcal{I}_{l}(1357, 8)$ to $\mathcal{I}_{u}(1357, 8)$.
\end{enumerate}
\end{example}

We now show the result for odd dimensions.

\begin{theorem}\label{thm-odd-quot}
Given triangulations $\mathcal{T}, \mathcal{T}'$ of $C(n, 2d + 1)$ such that $\mathcal{T} \lessdot \mathcal{T}'$, there exist cubillages $\mathcal{Q}_{0}, \dots, \mathcal{Q}_{r}$ of $Z(n, 2d + 2)$ such that $\mathcal{Q}_{0} = \mathcal{Q}_{\mathcal{T}}$, $\mathcal{Q}_{r} = \mathcal{Q}_{\mathcal{T}'}$ and \[\mathcal{Q}_{0} \lessdot \mathcal{Q}_{1} \lessdot \dots \lessdot \mathcal{Q}_{r},\] so that $\mathcal{Q}_{\mathcal{T}} \leqslant \mathcal{Q}_{\mathcal{T}'}$.
\end{theorem}
\begin{proof}
We start, as in the proof of Theorem~\ref{thm-g-odd}, by considering the triangulations $\hat{\mathcal{T}}, \hat{\mathcal{T}'}$ of $C(n+1,2d+2)$. By \cite[Proposition 5.14(i)]{rambau}, we have that $\hat{\mathcal{T}'} < \hat{\mathcal{T}}$. By Theorem~\ref{thm-even-quot}, there exist cubillages $\mathcal{Q}'_{s} \lessdot \dots \lessdot \mathcal{Q}'_{0}$ of $Z(n + 1, 2d + 3)$ such that $\mathcal{Q}'_{s} = \mathcal{Q}_{\hat{\mathcal{T}'}}$ and $\mathcal{Q}'_{0} = \mathcal{Q}_{\hat{\mathcal{T}}}$.

As in the proof of \cite[Lemma 5.2]{dkk}, we have that the $(n+1)$-contraction of $\mathcal{Q}'_{i}$ gives a membrane $\mathcal{M}_{i}$, which is a cubillage of $Z(n, 2d + 2)$. As in the proof of Theorem~\ref{thm-g-odd}, we have that $\mathcal{M}_{s} = \mathcal{Q}_{\mathcal{T}'}$ and $\mathcal{M}_{0} = \mathcal{Q}_{\mathcal{T}}$. We claim that for each $i$ we either have $\mathcal{M}_{i} = \mathcal{M}_{i + 1}$ or $\mathcal{M}_{i} \lessdot \mathcal{M}_{i + 1}$.

Consider the increasing flip which takes $\mathcal{Q}'_{i + 1}$ to $\mathcal{Q}'_{i}$. Suppose this increasing flip is induced by a $(2d + 4)$-face $\Gamma$ of $Z(n + 1, n + 1)$ which has $A$ as its set of generating vectors. If $n + 1 \notin A$, then the increasing flip does not affect the $(n + 1)$-pie, so that $\mathcal{M}_{i} = \mathcal{M}_{i + 1}$. Hence, suppose instead that $n + 1 \in A$. Let the lower facets of $\pi_{n + 1, 2d + 4}(\Gamma)$ consist of the cubes $\pi_{n + 1, 2d + 4}(\Delta_{j})$, where $\Delta_{j}$ is generated by $A \setminus \{a_{j}\}$, noting that we must have $a_{2d + 3} = n + 1$. Similarly, let the upper facets of $\pi_{n + 1, 2d + 4}(\Gamma)$ consist of the cubes $\pi_{n + 1, 2d + 4}(\Delta'_{j})$, where $\Delta'_{j}$ is generated by $A \setminus \{a_{j}\}$.

Then, it is well-known that for $j < k$ the cubes $\pi_{n + 1, 2d + 4}(\Delta_{j})$ and $\pi_{n + 1, 2d + 4}(\Delta_{k})$ intersect in an upper facet of $\pi_{n + 1, 2d + 4}(\Delta_{k})$ and a lower facet of $\pi_{n + 1, 2d + 4}(\Delta_{j})$, while the cubes $\pi_{n + 1, 2d + 4}(\Delta'_{j})$ and $\pi_{n + 1, 2d + 4}(\Delta'_{k})$ intersect in an upper facet of $\pi_{n + 1, 2d + 4}(\Delta'_{j})$ and a lower facet of $\pi_{n + 1, 2d + 4}(\Delta'_{k})$. This is because the increasing flip corresponds to inverting the packet of $A$: the cubes $\Delta_{j}$ and $\Delta'_{j}$ correspond to the sets $A \setminus \{a_{j}\}$; these must be ordered lexicographically for $\Delta_{j}$ and reverse-lexicographically for $\Delta'_{j}$.

Contracting the $(n + 1)$-pie of $\mathcal{Q}'_{i + 1}$ sends the cubes $\Delta_{j}$ for $j < 2d + 3$ to their facet generated by $A\setminus \{a_{j}, n + 1\}$, which is precisely the intersection $\Delta_{j} \cap \Delta_{2d + 3}$. By the above paragraph, this projects to an upper facet of $\pi_{n + 1, 2d + 4}(\Delta_{2d + 3})$. Hence the part of $\mathcal{M}_{i + 1}$ which lies within $\Gamma/(n + 1)$ consists of the upper facets of $\pi_{n + 1, 2d + 4}(\Delta_{2d + 3}/(n + 1))$. Here we use $\Gamma/(n + 1)$ to denote the image of $\Gamma/(n + 1)$ under the $(n + 1)$-contraction, and so forth. Similarly, we have that the part of $\mathcal{M}_{i}$ which lies within $\Gamma/(n + 1)$ consists of the lower facets of $\pi_{n + 1, 2d + 4}(\Delta'_{2d + 3}/(n + 1))$. We then have that $\Gamma/(n + 1) = \Delta_{2d + 3}/(n + 1) = \Delta'_{2d + 3}/(n + 1)$, and so $\mathcal{M}_{i} \lessdot \mathcal{M}_{i + 1}$. This is since $\mathcal{M}_{i}$ and $\mathcal{M}_{i + 1}$ only differ within $\Gamma/(n + 1)$, because $\mathcal{Q}'_{i + 1}$ and $\mathcal{Q}'_{i}$ only differ within $\Gamma$. Moreover, $\pi_{n, 2d + 3}(\mathcal{M}_{i + 1})$ contains the upper facets of $\pi_{n, 2d + 3}(\Gamma/(n + 1))$, whereas $\pi_{n, 2d + 3}(\mathcal{M}_{i})$ contains the lower facets of $\pi_{n, 2d + 3}(\Gamma/(n + 1))$. This argument is illustrated in Figure~\ref{fig:arg_ill}; compare \cite[Figure 7]{dkk-survey}.

This gives a chain of cubillages $\mathcal{Q}_{\mathcal{T}} = \mathcal{M}_{0} = \mathcal{Q}_{0} \lessdot \dots \lessdot \mathcal{Q}_{r} = \mathcal{M}_{s} = \mathcal{Q}_{\mathcal{T}'}$ by applying the result of the above paragraph to the chain $\mathcal{Q}'_{s} \lessdot \dots \lessdot \mathcal{Q}'_{0}$. Here the cubillages $\mathcal{Q}_{0}, \dots, \mathcal{Q}_{r}$ are the cubillages $\mathcal{M}_{0}, \dots, \mathcal{M}_{s}$ with the duplicates removed, corresponding to the cases above where $\mathcal{M}_{i} = \mathcal{M}_{i + 1}$.
\end{proof}

\begin{figure}
\caption{An illustration of the argument of Theorem~\ref{thm-odd-quot}.}\label{fig:arg_ill}
\[
\scalebox{0.8}{
\begin{tikzpicture}

\begin{scope}[xscale=0.7]

% Boundary vertices

\coordinate(0) at (0,0);
\node at (0)[left = 1mm of 0]{$\emptyset$};
\coordinate(1) at (2,-2);
\node at (1)[below left = 1mm of 1]{1};
\coordinate(12) at (4,-3);
\node at (12)[below = 1mm of 12]{12};
\coordinate(123) at (6,-2);
\node at (123)[below right = 1mm of 123]{123};
\coordinate(1234) at (8,0);
\node at (1234)[right = 1mm of 1234]{1234};
\coordinate(234) at (6,2);
\node at (234)[above right = 1mm of 234]{234};
\coordinate(34) at (4,3);
\node at (34)[above = 1mm of 34]{34};
\coordinate(4) at (2,2);
\node at (4)[above left = 1mm of 4]{4};

% Boundary edges

\draw (0) -- (1) -- (12) -- (123) -- (1234) -- (234) -- (34) -- (4) -- (0);

% Internal vertices

\coordinate(3) at (2,1);
\coordinate(13) at (4,-1);
\coordinate(23) at (4,0);

% Filled cubes

\draw[fill=red!30,draw=none] (0) -- (4) -- (34) -- (3) -- (0);
\draw[fill=red!30,draw=none] (3) -- (34) -- (234) -- (23) -- (3);
\draw[fill=red!30,draw=none] (23) -- (234) -- (1234) -- (123) -- (23);

% Internal edges

\draw (0) -- (3);
\draw (3) -- (34);
\draw (3) -- (23);
\draw (3) -- (13);
\draw (1) -- (13);
\draw (13) -- (123);
\draw (23) -- (123);
\draw (23) -- (234);

% Labels of internal vertices

\node at (3) [below = 1mm of 3]{3};
\node at (13) [left = 1mm of 13]{13};
\node at (23) [below = 1mm of 23]{23};

% Label of cubillage

\node at (-1.5,0) {\huge $\mathcal{Q}_{i+1}$};

%%%%%%%%%%%%%%%%%%%%%%%%%%%%%%%%%%%%%%%%%%%%
% CONTRACTED VERSION
%%%%%%%%%%%%%%%%%%%%%%%%%%%%%%%%%%%%%%%%%%%%

% Arrow from left to right

\draw[->,ultra thick] (9.75, 0) -- (10.75, 0);

% Boundary vertices

\coordinate(0) at (12,0);
\node at (0)[left = 1mm of 0]{$\emptyset$};
\coordinate(1) at (14,-2);
\node at (1)[below left = 1mm of 1]{1};
\coordinate(12) at (16,-3);
\node at (12)[below = 1mm of 12]{12};
\coordinate(123) at (18,-2);
\node at (123)[below right = 1mm of 123]{123};
%\coordinate(1234) at (20,0);
%\node at (1234)[right = 1mm of 1234]{1234};
%\coordinate(234) at (18,2);
%\node at (234)[above right = 1mm of 234]{234};
%\coordinate(34) at (16,3);
%\node at (34)[above = 1mm of 34]{34};
%\coordinate(4) at (14,2);
%\node at (4)[above left = 1mm of 4]{4};

% Internal vertices

\coordinate(3) at (14,1);
\coordinate(13) at (16,-1);
\coordinate(23) at (16,0);

% Boundary edges

\draw (0) -- (1) -- (12) -- (123) -- (23) -- (3) -- (0);

%(1234) -- (234) -- (34) -- (4) -- (0);

% Internal edges

%\draw (0) -- (3);
%\draw (3) -- (34);
%\draw (3) -- (23);
\draw (3) -- (13);
\draw (1) -- (13);
\draw (13) -- (123);
%\draw (23) -- (123);
%\draw (23) -- (234);

% Labels of internal vertices

\node at (3) [above left = 1mm of 3]{3};
\node at (13) [left = 1mm of 13]{13};
\node at (23) [right = 1mm of 23]{23};

% Membrane

\draw[red,ultra thick] (0) -- (3) -- (23) -- (123);

% Label of membrane

\node at (18,0) {\huge \color{red} $\mathcal{M}_{i+1}$};

\end{scope}

\end{tikzpicture}
}
\]
\[
\scalebox{0.8}{
\begin{tikzpicture}

\begin{scope}[xscale=0.7]

% Boundary vertices

\coordinate(0) at (0,0);
\node at (0)[left = 1mm of 0]{$\emptyset$};
\coordinate(1) at (2,-2);
\node at (1)[below left = 1mm of 1]{1};
\coordinate(12) at (4,-3);
\node at (12)[below = 1mm of 12]{12};
\coordinate(123) at (6,-2);
\node at (123)[below right = 1mm of 123]{123};
\coordinate(1234) at (8,0);
\node at (1234)[right = 1mm of 1234]{1234};
\coordinate(234) at (6,2);
\node at (234)[above right = 1mm of 234]{234};
\coordinate(34) at (4,3);
\node at (34)[above = 1mm of 34]{34};
\coordinate(4) at (2,2);
\node at (4)[above left = 1mm of 4]{4};

% Boundary edges

\draw (0) -- (1) -- (12) -- (123) -- (1234) -- (234) -- (34) -- (4) -- (0);

% Internal vertices

\coordinate(3) at (2,1);
\coordinate(13) at (4,-1);
\coordinate(134) at (6,1);

% Filled cubes

\draw[fill=red!30,draw=none] (0) -- (4) -- (34) -- (3) -- (0);
\draw[fill=red!30,draw=none] (3) -- (34) -- (134) -- (13) -- (3);
\draw[fill=red!30,draw=none] (13) -- (134) -- (1234) -- (123) -- (13);

% Internal edges

\draw (0) -- (3);
\draw (3) -- (34);
\draw (13) -- (134);
\draw (34) -- (134);
\draw (3) -- (13);
\draw (1) -- (13);
\draw (13) -- (123);
\draw (134) -- (1234);

% Labels of internal vertices

\node at (3) [below = 1mm of 3]{3};
\node at (13) [left = 1mm of 13]{13};
\node at (134) [above = 1mm of 134]{134};

% Label of cubillage

\node at (-1.5,0) {\huge $\mathcal{Q}_{i}$};

%%%%%%%%%%%%%%%%%%%%%%%%%%%%%%%%%%%%%%%%%%%%
% CONTRACTED VERSION
%%%%%%%%%%%%%%%%%%%%%%%%%%%%%%%%%%%%%%%%%%%%

% Arrow from left to right

\draw[->,ultra thick] (9.75, 0) -- (10.75, 0);

% Boundary vertices

\coordinate(0) at (12,0);
\node at (0)[left = 1mm of 0]{$\emptyset$};
\coordinate(1) at (14,-2);
\node at (1)[below left = 1mm of 1]{1};
\coordinate(12) at (16,-3);
\node at (12)[below = 1mm of 12]{12};
\coordinate(123) at (18,-2);
\node at (123)[below right = 1mm of 123]{123};
%\coordinate(1234) at (20,0);
%\node at (1234)[right = 1mm of 1234]{1234};
%\coordinate(234) at (18,2);
%\node at (234)[above right = 1mm of 234]{234};
%\coordinate(34) at (16,3);
%\node at (34)[above = 1mm of 34]{34};
%\coordinate(4) at (14,2);
%\node at (4)[above left = 1mm of 4]{4};

% Internal vertices

\coordinate(3) at (14,1);
\coordinate(13) at (16,-1);
\coordinate(23) at (16,0);

% Boundary edges

\draw (0) -- (1) -- (12) -- (123) -- (23) -- (3) -- (0);

%(1234) -- (234) -- (34) -- (4) -- (0);

% Internal edges

%\draw (0) -- (3);
%\draw (3) -- (34);
%\draw (3) -- (23);
\draw (3) -- (13);
\draw (1) -- (13);
\draw (13) -- (123);
%\draw (23) -- (123);
%\draw (23) -- (234);

% Labels of internal vertices

\node at (3) [above left = 1mm of 3]{3};
\node at (13) [left = 1mm of 13]{13};
\node at (23) [right = 1mm of 23]{23};

% Membrane

\draw[red,ultra thick] (0) -- (3) -- (13) -- (123);

% Label of membrane

\node at (18,0) {\huge \color{red} $\mathcal{M}_{i}$};

\end{scope}

\end{tikzpicture}
}
\]
\end{figure}
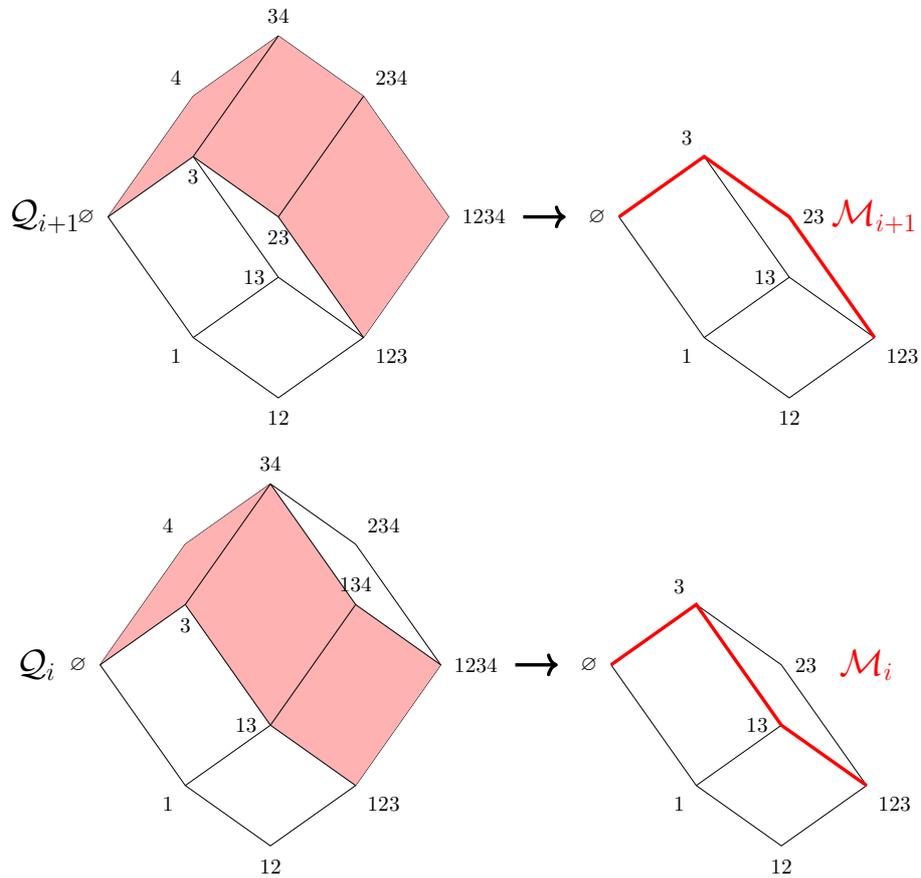

By putting together Theorem~\ref{thm-g-even}, Theorem~\ref{thm-g-odd}, Theorem~\ref{thm-even-quot}, and Theorem~\ref{thm-odd-quot}, this finally establishes Theorem~\ref{thm:quot}, and hence also Corollary~\ref{cor:t=st}.

\printbibliography

\end{document}